\documentclass{amsart}

\usepackage{amsmath,amsthm,amssymb}
\usepackage{amsmath}
\usepackage{amsfonts}
\usepackage{cancel}
\usepackage[inline]{asymptote}
\usepackage[margin=1in]{geometry}
\usepackage{hyperref}

\makeatletter
\renewcommand{\pod}[1]{\allowbreak\mathchoice
  {\if@display \mkern 18mu\else \mkern 8mu\fi (#1)}
  {\if@display \mkern 18mu\else \mkern 8mu\fi (#1)}
  {\mkern4mu(#1)}
  {\mkern4mu(#1)}
}

\renewcommand{\Re}{\mathrm{Re}}
\renewcommand{\Im}{\mathrm{Im}}

\newtheorem{theorem}{Theorem}[subsection]
\newtheorem{lemma}[theorem]{Lemma}
\newtheorem{corollary}[theorem]{Corollary}
\theoremstyle{remark}
\newtheorem{remark}[theorem]{Remark}
\numberwithin{equation}{subsection}
\newcommand{\cond}{\operatorname{cond}}

\newcommand{\flrt}{\operatorname{flrt}}

\title{\textbf{Low-Lying Zeros of a Thin Family of Automorphic $L$-Functions in the Level Aspect}}
\author{Matthew Kroesche}
\email{mdkroesche@tamu.edu}

\begin{document}

\maketitle
{\small \textsc{Abstract.} We calculate the one-level density of thin subfamilies of a family of Hecke cuspforms formed by twisting the forms in a smaller family by a character. The result gives support up to 1, conditional on GRH, and we also find several of the lower-order main terms. In addition, we find an unconditional result that has only slightly lower support. A crucial step in doing so is the establishment of an on-average version of the Weil bound that applies to twisted Kloosterman sums. Moreover, we average over these thin subfamilies by running over the characters in a coset, and observe that any amount of averaging at all is enough to allow us to get support greater than 1 and thus distinguish between the SO(even) and SO(odd) symmetry types. Finally, we also apply our results to nonvanishing problems for the families studied.}

\begin{section}{Introduction}

\begin{subsection}{Background}
Suppose $\phi$ is an even Schwartz-class function, and $f$ is any Hecke cuspform. We define the \emph{one-level density} of $f$ with respect to $\phi$ as
\begin{equation}
D_1(f, \phi, R) = \sum_{\rho}\phi\left(\frac{\log R\left(\rho-\frac12\right)}{2\pi i}\right),\label{olddef}
\end{equation}
where the sum is over the nontrivial zeros $\rho$ of $L(f, s)$, and $R$ is a scaling parameter related to the conductor of $L(f, s)$. Throughout this work, we consider forms of level $q^2$ and use the value $R=\left(\frac{q}{2\pi}\right)^2$. We also define the \emph{average one-level density} of any family $\mathcal{F}$ of Hecke cuspforms as
\begin{equation}
\mathcal{D}_1(\mathcal{F},\phi,R, w) = \frac{\sum_{f\in\mathcal{F}}w(f)D_1(f, \phi, R)}{\sum_{f\in\mathcal{F}}w(f)},\label{oldavgdef}
\end{equation}
where $w(f)$ is any positive weight function on the family. \\

In \cite{ILS}, the average one-level density of the family $\mathcal{H}_{\kappa}(q, 1)$ of newforms of level $q$, weight $\kappa$, and trivial central character was calculated in level aspect as $q\to\infty$, valid for test functions $\phi$ with Fourier transform supported in $(-2, 2)$. This family has orthogonal symmetry. The authors also considered the subfamilies of modular forms whose $L$-functions have only positive or only negative root number, and noted that it is possible to detect the SO(even) or SO(odd) symmetry of these subfamilies. \\

Another key work in the study of low-lying zeros is \cite{HR}, which established that the symmetry type of the family of Dirichlet $L$-functions of a particular conductor is unitary, and also calculated several other statistics of the zero distribution in analogy with the Katz-Sarnak philosophy (introduced in \cite{KS}). The subfamily of Dirichlet $L$-functions of a quadratic character was later studied in \cite{OS}, and in particular it has been noted that this subfamily has symplectic symmetry. \\
\end{subsection}

\begin{subsection}{Statement of Results}
In this paper, we establish results for smaller subfamilies of the family studied in \cite{ILS}. Due to the reduced size of the families, it is not surprising that the support will be less than the support found in \cite{ILS}. Our most basic object of study is a thin subfamily that was introduced in \cite{PY} and used by the authors to establish that the Weyl bound holds for Dirichlet $L$-functions of cubefree conductor (and later, for all Dirichlet $L$-functions). The size of this subfamily is proportional in the limit to the square root of the size of the whole family of modular forms. Our first main result is to calculate the one-level density of this thin subfamily (and twists of it by characters) for test functions supported in $(-1, 1)$ and observe that it matches the symmetry type of the large family within this restricted support. Our second main result is to average this subfamily by twisting it over a coset to build larger subfamilies of an intermediate size, and investigate how large the support is allowed to be as a function of the size of the coset. The final result in this case is essentially the best possible: the support linearly interpolates the size (in the $p$-adic sense) of the coset from $(-1, 1)$ up to $(-2, 2)$. In particular, any nonzero amount of averaging is enough to break past $(-1, 1)$ and detect whether the symmetry type of the subfamily is even or odd. \\

Let $\mathcal{H}_{\kappa}(q, \chi)$ denote the set of holomorphic Hecke newforms of weight $\kappa$, level $q$, and central character $\chi$. Suppose $\kappa$ is even and $q$ is odd. Then for Dirichlet characters $\chi$ and $\eta$ modulo $q$ such that $\chi^2$, $\chi\eta$, and $\chi\overline{\eta}$ all have conductor $q$, we are interested in the family of twists
\begin{equation}
\mathcal{F}_{\kappa}(q,\chi,\eta) = \{f_{\chi\eta} : f\in\mathcal{H}_{\kappa}(q,\overline{\chi}^2)\},\label{family1}
\end{equation}
where $f_{\chi\eta}$ is shorthand for $f\otimes \chi\otimes\eta$. We claim that the family $\mathcal{F}_{\kappa}(q,\chi,\eta)$ consists only of newforms, and is thus a subfamily of $\mathcal{H}_{\kappa}(q^2, \eta^2)$ of size $(\kappa q)^{1+o(1)}$. Indeed, due to Lemma 1.4 of \cite{BLS}, we have $\cond(f_{\chi\eta}) = q^2$ as well, using our assumptions that $\chi\eta$ and $\overline{\chi}\eta$ are primitive modulo $q$. We also note that, since $\chi^2$ is primitive, there are no oldforms of weight $\kappa$, level $q$, and central character $\overline{\chi}^2$. Thus $\mathcal{H}_{\kappa}(q,\overline{\chi}^2)$ is an orthonormal basis for the space of cuspforms $\mathcal{S}_{\kappa}(q,\overline{\chi}^2)$, so we may average over it using the Petersson formula without any sieving. Our first result is an asymptotic for the average one-level density of this family weighted using the Petersson weights of the forms $f$.

\begin{theorem}[One-Level Density of Thin Family]
Suppose $\widehat{\phi}$ has support contained in $[-\theta,\theta]$.
For any cuspform $f$ of level $q$, define the Petersson weight function by
\begin{equation}
w(f) = \frac1{\left<f,f\right>_q},\label{petwt}
\end{equation}
where $\left<\cdot,\cdot\right>_q$ is the Petersson inner product as defined in \eqref{petprod}, and take $\mathcal{F} = \mathcal{F}_{\kappa}(q,\chi,\eta)$. Let the integrals $\mathcal{I}(\kappa,\phi,R)$, $\mathcal{J}(q,\phi,R)$, and $\mathcal{J}(q, \phi,R,\chi)$ be as defined in \eqref{idef}, \eqref{jdef}, and \eqref{jgendef} respectively, and let the weight functions $W(G)(x)$ be as defined in \eqref{wux} through \eqref{wsoox}. Then if $\eta$ is quadratic, we have
\begin{multline}
\mathcal{D}_1\left(\mathcal{F}, \phi, R, w\right) = \int_{-\infty}^{\infty}\phi(x)W(\mathrm{O})(x)\;dx \\
+ \frac2{\log R}\widehat{\phi}(0) + \frac2{\log R}\mathcal{I}(\kappa,\phi, R) - \frac2{\log R}\mathcal{J}(q, \phi, R) + O\left(||\widehat{\phi}||_{\infty}\left(q^{\kappa(\theta-1)+\frac12+\varepsilon} + q^{-1+\varepsilon}\right)\right).\label{result1a}
\end{multline}
If $\eta$ is not quadratic, we instead have
\begin{multline}
\mathcal{D}_1\left(\mathcal{F}, \phi, R, w\right) = \int_{-\infty}^{\infty}\phi(x)W(\mathrm{U})(x)\;dx \\
+ \frac2{\log R}\mathcal{I}(\kappa,\phi, R) - \frac2{\log R}\mathcal{J}(q, \phi, R, \eta^2)
+ O\left(||\widehat{\phi}||_{\infty}\left(q^{\kappa(\theta-1)+\frac12+\varepsilon} + q^{-1+\varepsilon}\right)\right).\label{result1b}
\end{multline}

The implied constants depend only on $\kappa$ and $\varepsilon$. Assuming the Generalized Riemann Hypothesis (GRH), the error terms in \eqref{result1a} and \eqref{result1b} can be improved for $\theta$ close to 1 to
\begin{equation}
O\left(\left(||\widehat{\phi}||_{\infty} + ||\widehat{\phi}\thinspace''||_{\infty}\right)q^{\frac{\theta-1}2+\varepsilon}\right).\label{result1e}
\end{equation}\label{theorem1}
\end{theorem}

This theorem is proved in Section \ref{thinsection}.

\begin{remark}
Throughout this paper, we are working in \emph{level aspect}, meaning that we think of $q$ as going to infinity and $\kappa$ as being a fixed positive even integer. Thus, many of the error terms in our work have an implicit dependence on $\kappa$.
\end{remark}

\begin{remark}
Throughout the paper, we consider $\theta$ as bounded by an absolute constant. For instance, our results are trivial if $\theta > 2$. As a consequence, none of our implied constants, such as in \eqref{result1a} or \eqref{result1b}, will depend on $\theta$.
\end{remark}

\begin{remark}
The error term gives unconditional support for $\widehat{\phi}$ up to $1 - \frac1{2\kappa}$. Conditional on the Generalized Riemann Hypothesis, this can be improved to 1.
\end{remark}

\begin{remark}
If $q$ is sent to infinity, all the lower-order main terms go to zero. Thus we can write the much more succinct result
\begin{equation}
\lim_{q\to\infty}\mathcal{D}_1(\mathcal{F},\phi,R,w) = \int_{-\infty}^{\infty}\phi(x)W(G)(x)\;dx,
\end{equation}
where $G$ is the symmetry type of the family; that is, $G = \mathrm{O}$ (orthogonal symmetry) if $\eta^2 = 1$ and $G = \mathrm{U}$ (unitary symmetry) otherwise.
\end{remark}

\begin{remark}
Lower-order main terms similar to these have appeared elsewhere in the literature of low-lying zeros as well. For example, an analogue of the integral $\mathcal{I}(\kappa,\phi, R)$ appears in (2.2) of \cite{MI} and (1.3) of \cite{RR}. Similarly, an integral analogous to $\mathcal{J}(q,\phi,R)$ appears as the lower-order term $S_3$ in Proposition 3.1 of \cite{RR}. 
\end{remark}

Our second result is an asymptotic for the one-level density of an intermediate coset family. Suppose $p$ is an odd prime and $q = p^k$ is a prime power with $k\ge 2$, and let $\psi$ be a Dirichlet character modulo $q$ such that $\cond(\psi^2) = q$. Let $j$ be a positive integer with $1\le j < k$, and let $\widehat{G}_N$ denote the group of Dirichlet characters modulo $N$ for any positive integer $N$. Suppose also that $\epsilon = \pm 1$. Then we define
\begin{equation}
\mathcal{F}_{\kappa,\epsilon}(q,p^j,\psi) = \bigcup_{\substack{\chi\in \psi\widehat{G}_{p^j} \\ \chi(-1) = \epsilon}}\mathcal{F}_{\kappa}(q,\chi,1).
\end{equation}
Several remarks are in order here. First, this is a subfamily of $\mathcal{H}_{\kappa}(q^2,1) = \mathcal{H}_{\kappa}(p^{2k}, 1)$ whose size is proportional to $p^{j+k}$, which is larger than $q$ but smaller than $q^2$. As before, we claim it contains only newforms. Indeed, every character in the coset $\psi\widehat{G}_{p^j}$ of $\widehat{G}_q$ has conductor $q$, because $\psi$ has conductor $q$, but the characters in $\widehat{G}_{p^j}$ have conductor dividing $p^j$, which is strictly less than $q$. Additionally, all the forms in the family have root number $i^{\kappa}\epsilon$ (see Section 2.3 of \cite{PY}), so the whole family is either even or odd, depending on whether $\kappa$ is $0\pmod 4$ or $2\pmod 4$, and on the sign $\epsilon$. Finally, we have taken $\eta = 1$ here for simplicity.

\begin{theorem}[One-Level Density of Coset Family]
Assume the Generalized Riemann Hypothesis, suppose $\widehat{\phi}$ has support contained in $[-\theta,\theta]$, and let $w(f)$ be the Petersson weight defined in \eqref{petwt}. Then if $q = p^k$, $\psi$ is a Dirichlet character modulo $q$ with $\psi^2$ primitive, and $\mathcal{F} = \mathcal{F}_{\kappa,\epsilon}(q,p^j,\psi)$, we have
\begin{multline}
\mathcal{D}_1\left(\mathcal{F}, \phi, R, w\right) = \int_{-\infty}^{\infty}\phi(x)W(G)(x)\;dx + \frac2{\log R}\widehat{\phi}(0) \\
+ \frac2{\log R}\mathcal{I}(\kappa,\phi, R) - \frac2{\log R}\mathcal{J}(q, \phi, R) - \frac{2i^{\kappa}\epsilon}{\log R}\mathcal{L}(\kappa,p,\phi, R) + O\left(\left(||\widehat{\phi}||_{\infty} + ||\widehat{\phi}\thinspace''||_{\infty}\right)q^{\max(\theta,1)-1-\frac{j}k+\varepsilon}\right),\label{result2}
\end{multline}
where $G$ is $\mathrm{SO}(\mathrm{even})$ if $i^{\kappa}\epsilon = 1$ and $\mathrm{SO}(\mathrm{odd})$ if $i^{\kappa}\epsilon = -1$, and $\mathcal{I}(\kappa,\phi,R)$, $\mathcal{J}(q,\phi,R)$, and $\mathcal{L}(\kappa,p,\phi,R)$ are integrals involving $\phi$ to be defined in \eqref{idef}, \eqref{jdef}, and \eqref{ldef} respectively. The implied constant depends only on $\kappa$ and $\varepsilon$.\label{theorem2}
\end{theorem}

This theorem is proved in Section \ref{cosetsection}.

\begin{remark}
The leading term tells us that $\mathcal{F}_{\kappa,\epsilon}(q,p^j,\psi)$ has special orthogonal even/odd symmetry, depending on the root number of the cuspforms in the family.
\end{remark}

\begin{remark}
The distributions $W(\mathrm{O})(x)$, $W(\mathrm{SO(even)})(x)$, and $W(\mathrm{SO(odd)})(x)$ all appear identical if the support of $\widehat{\phi}$ is contained in $(-1, 1)$, since in all three cases we have
\begin{equation}
\int_{-\infty}^{\infty}\phi(x)W(G)(x)\;dx = \widehat{\phi}(0) + \frac12\phi(0).
\end{equation}
As a consequence, our result \eqref{result1a} is unable to distinguish between even and odd orthogonal symmetry. However, the additional averaging we do in our coset family allows our support to break out of $(-1, 1)$, in fact up to $1 + \frac{j}k$, and this lets us distinguish between these symmetry types. In fact, just picking $j = \left\lfloor\varepsilon k\right\rfloor$ is sufficient to let us distinguish between SO(even) and SO(odd).
\end{remark}

\begin{remark}
A possible direction of further study might be to relax the restriction that $\eta = 1$ in the coset family. We expect that this should give the similar result that the symmetry type becomes only unitary if $\eta^2 \neq 1$.
\end{remark}

\begin{remark}
Another possible direction of study might be to calculate an error term for the coset family that is unconditional on GRH, likely by establishing an on-average Weil-type bound for the coset sum $\sigma_{\psi, \epsilon}(m,n;c)$ defined in \eqref{cosetsum}. As in the thinnest family case, this error term might not be able to give us support all the way up to $1 + \frac{j}k$.
\end{remark}

\begin{remark}
Much of the work on low-lying zeros of modular forms in the literature studies the weight aspect rather than the level aspect. In particular, \cite{ILS} gives multiple results in both the weight and level aspect, as well as some hybrid results. In the weight aspect, the average is taken over modular forms of different weights within a certain range $[T, T+\Delta]$ going off to infinity, and the level is fixed (usually 1). The parallels between the two aspects have long been considered as well (also going back to \cite{ILS}) but in the case of the coset family, the analogy is particularly striking. We note that $T$ and $\Delta$ in the weight aspect are analogous to $p^k$ and $p^j$ respectively in the level aspect, since $T$ (like $p^k$) is the ``starting point" for the weight, and $\Delta$ (like $p^j$) controls the size of the family over which the average is taken. Similarly, the case $\Delta \to 1$ is analogous to the case $j\to 0$ of our thinnest family, and this family in weight aspect is studied in \cite{DFS}. Other relevant works in the literature include \cite{RI} and \cite{FR}, which calculate moments of $L$-functions attached to holomorphic cuspforms in the weight apsect. (The former studies the fifth moment over a range of weights, while the latter studies the cubic moment for an individual weight.) We can also find analogies between the character sums we study in this work, and sums that arise in the weight aspect as well, culminating in a decomposition in the weight aspect that very closely mirrors our results in Section \ref{cosetsumsection} about when the same-sign and opposite-sign parts of the coset sum, as given in \eqref{sigmasplit}, vanish. The details of this analogy are studied in Section \ref{analogysection}.
\end{remark}

\begin{remark}
We observe that every cuspform in $\mathcal{F}_{\kappa}(q,\chi,\eta)$ has the property that the local representation at every ramified prime is a principal series. This was established in Section 1.5 of \cite{PY2}.
\end{remark}

\begin{remark}
It is also possible to modify this approach to study families of Hecke-Maass cuspforms by using the Kuznetsov trace formula instead of the Petersson trace formula. This was notably done in \cite{AAILMZ} and \cite{AM}, both of which established analogues of many of the results about low-lying zeros in \cite{ILS} for families of Maass cuspforms. Another important result in this vein is \cite{PY}, which establishes the cubic moment of both the thin holomorphic family of modular forms studied in this paper and its equivalent family of Hecke-Maass cuspforms.
\end{remark}

\begin{remark}
The fact that this subfamily was used to establish the Weyl bound for Dirichlet $L$-functions in \cite{PY} and \cite{PY2} is the primary motivation for our study of its low-lying zeros in this paper. It also motivates questions about other related statistics of the subfamily: for example, its cubic moment was found in \cite{PY}, but its other moments have not yet been studied.
\end{remark}

We also may apply these results to questions of nonvanishing of central values, following the work of \cite{ILS}. For any family $\mathcal{F}$, any weight function $w : \mathcal{F} \to (0, \infty)$, and for a nonnegative integer $m$ we define
\begin{equation}
p_m(\mathcal{F}, w) = \frac1{\sum_{f\in\mathcal{F}}w(f)}\displaystyle\sum_{\substack{f\in\mathcal{F} \\ \operatorname{ord}_{s=\frac12}L(s,f)=m}}w(f)
\end{equation}
to be the weighted proportion of the elements of $\mathcal{F}$ whose $L$-function vanishes with order exactly $m$ at the central point. In particular, if $w$ is constant, then $p_m(\mathcal{F}, w)$ gives the proportion of elements in $\mathcal{F}$ whose $L$-function vanishes at the central point with order $m$. Then we can state the following results for our family:
\begin{theorem}
Assume the Generalized Riemann Hypothesis. Let $\mathcal{F} = \mathcal{F}_{\kappa}(q,\chi,\eta)$ where $\eta$ is a quadratic character modulo $q$, and let $w(f) = \frac1{\left<f,f\right>_{q^2}}$ be the usual Petersson weight. Let $\varepsilon > 0$. Then if $i^{\kappa}\chi(-1) = 1$, we have
\begin{equation}
p_0(\mathcal{F}, w) \ge \frac14 - \varepsilon,\label{thinevennv}
\end{equation}
for sufficiently large $q$. If instead $i^{\kappa}\chi(-1) = -1$, then we have $p_0(\mathcal{F}, w) = 0$ and
\begin{equation}
p_1(\mathcal{F}, w) \ge \frac34 - \varepsilon,\label{thinoddnv}
\end{equation}
for sufficiently large $q$.\label{thintheoremnv}
\end{theorem}
\begin{theorem}
Assume the Generalized Riemann Hypothesis. Let $q = p^k$, let $\mathcal{F} = \mathcal{F}_{\kappa,\epsilon}(q,p^j,\psi)$, and let $w(f) = \frac1{\left<f,f\right>_{q^2}}$ be the usual Petersson weight. Let $\varepsilon > 0$. Then if $i^{\kappa}\epsilon = 1$, we have
\begin{equation}
p_0(\mathcal{F}, w) \ge \left(1 - \frac{k}{2(j+k)}\right)^2 - \varepsilon,\label{cosetevennv}
\end{equation}
for sufficiently large $q$. If instead $i^{\kappa}\epsilon = -1$, then we have $p_0(\mathcal{F}, w) = 0$ and
\begin{equation}
p_1(\mathcal{F}, w) \ge 1 - \left(\frac{k}{2(j+k)}\right)^2 - \varepsilon,\label{cosetoddnv}
\end{equation}
for sufficiently large $q$.\label{cosettheoremnv}
\end{theorem}
\begin{remark}
If $j=k$, we observe that \eqref{cosetevennv} and \eqref{cosetoddnv} agree with equations (1.48) and (1.49) in \cite{ILS}.
\end{remark}

These two theorems are proved in Section \ref{nvsection}.

\end{subsection}
\begin{subsection}{The Weight Function}
Here we consider the relationship between the Petersson weight $w(f) = \frac1{\left<f,f\right>_q}$ of a cuspform and the weight $w(f_{\chi\eta}) = \frac1{\left<f_{\chi\eta},f_{\chi\eta}\right>_{q^2}}$ of its twist. We have the following result:
\begin{lemma}
Suppose $f\in\mathcal{H}_{\kappa}(q,\overline{\chi}^2)$, and $\eta$ is such that $\chi\eta$ and $\overline{\chi}\eta$ are primitive to the modulus $q$. Then
\begin{equation}
w(f_{\chi\eta}) = \frac{q}{\varphi(q)}w(f).
\end{equation}
\end{lemma}
\begin{remark}
In particular, the ratio between the weights does not depend on $f$, so they give the same average one-level density. Thus we use $w(f)$ and $w(f_{\chi\eta})$ interchangeably throughout this paper.
\end{remark}

\begin{proof}
We use the Rankin-Selberg method. From the definition of the Petersson inner product given in \eqref{petprod}, we write
\begin{equation}
\left<fE_s, f\right>_q = \frac3{\pi\left[\Gamma_0(1): \Gamma_0(q)\right]}\iint_{\Gamma_0(q)\setminus\mathbb{H}}y^{\kappa}f(z)\overline{f(z)}E_s(z)\frac{dx\;dy}{y^2},
\end{equation}
where
\begin{equation}
E_s(z) = \sum_{\gamma\in\Gamma_{\infty}\setminus \Gamma_0(q)}\Im(\gamma z)^s
\end{equation}
is an Eisenstein series for $\Re(s) > 1$. Unfolding and then applying standard calculations gives
\begin{equation}
\left<fE_s, f\right>_q = \frac{12\Gamma(\kappa+s-1)}{(4\pi)^{\kappa+s}\left[\Gamma_0(1): \Gamma_0(q)\right]}\sum_{n\ge 1}\frac{|\lambda_f(n)|^2}{n^s}.
\end{equation}
Now $E_s(z)$ has a simple pole at $s=1$ with residue $\frac3{\pi\left[\Gamma_0(1):\Gamma_0(q)\right]}$, so we take the residue at $s=1$ on both sides to get
\begin{equation}
\left<f, f\right>_q = \frac{\Gamma(\kappa)}{(4\pi)^{\kappa}}\operatorname{Res}_{s=1}\left[\sum_{n\ge 1}\frac{|\lambda_f(n)|^2}{n^s}\right].
\end{equation}
Similarly, we apply this process to $f_{\chi\eta} \in \mathcal{H}_{\kappa}(q^2,\eta^2)$ to get
\begin{equation}
\left<f_{\chi\eta}, f_{\chi\eta}\right>_{q^2} = \frac{\Gamma(\kappa)}{(4\pi)^{\kappa}}\operatorname{Res}_{s=1}\left[\sum_{n\ge 1}\frac{|\lambda_{f_{\chi\eta}}(n)|^2}{n^s}\right].
\end{equation}
Now
\begin{equation}
|\lambda_{f_{\chi\eta}}(n)|^2 = |\chi(n)\eta(n)\lambda_f(n)|^2.
\end{equation}
This equals $|\lambda_f(n)|^2$ when $(n, q) = 1$ and $0$ when $(n, q) \neq 1$. So we have
\begin{equation}
\left<f_{\chi\eta}, f_{\chi\eta}\right>_{q^2} = \frac{\Gamma(\kappa)}{(4\pi)^{\kappa}}\operatorname{Res}_{s=1}\left[\sum_{\substack{n\ge 1 \\ (n,q)=1}}\frac{|\lambda_f(n)|^2}{n^s}\right].
\end{equation}
Now we write
\begin{equation}
\sum_{\substack{n\ge 1 \\ (n,q)=1}}\frac{|\lambda_f(n)|^2}{n^s} = \sum_{n\ge 1}\sum_{d\mid (n,q)}\frac{\mu(d)|\lambda_f(n)|^2}{n^s} = \sum_{d\mid q}\mu(d)\sum_{\substack{n\ge 1 \\ d\mid n}}\frac{|\lambda_f(n)|^2}{n^s} = \sum_{d\mid q}\frac{\mu(d)}{d^s}\sum_{n\ge 1}\frac{|\lambda_f(dn)|^2}{n^s}.
\end{equation}
We recall that $\lambda_f$ is completely multiplicative at ramified primes. Thus $\lambda_f(dn) = \lambda_f(d)\lambda_f(n)$. Moreover, $|\lambda_f(p)| = 1$ at ramified primes $p$ due to Theorem 1.1 of \cite{AL}, so $|\lambda_f(d)|^2 = 1$ and thus $|\lambda_f(dn)|^2 = |\lambda_f(n)|^2$. Thus we write
\begin{equation}
\left<f_{\chi\eta},f_{\chi\eta}\right>_{q^2} = \left<f,f\right>_q\cdot\sum_{d\mid q}\frac{\mu(d)}d,
\end{equation}
and the result follows.
\end{proof}

\end{subsection}
\begin{subsection}{The Weil Bound}
If we do not assume GRH, then we need a way to bound the twisted Kloosterman sums arising from the Petersson trace formula. (Conditional on GRH, we do not use a Weil-type bound, but rather the Fourier decomposition \eqref{klfourier}.) The Weil bound \eqref{weil} does \emph{not} hold in general for twisted Kloosterman sums (see Example 9.9 of \cite{KL} for an explicit counterexample). In \cite{KL}, the authors established a result, reproduced in \eqref{weilkl}, which is acceptable when $c$ is much larger than the conductor $q$ of the character, but very bad when the character is primitive. However, in this paper we establish an ``on-average" result that is almost as good as the classical Weil bound.

\begin{lemma}
Let $A, B$ be positive integers with $A < B$, and let $c$ be a positive integer and $\chi$ a Dirichlet character modulo $c$. Then
\begin{equation}
\sum_{\substack{m,n\ge 1 \\ A\le mn\le B}}|S_{\chi}(m,n;c)| \ll c^{\frac12+\varepsilon}B^{\varepsilon}(B-A+c^{\frac12}),\label{weilavg}
\end{equation}
where the implied constant depends only on $\varepsilon$.\label{weilavglemma}
\end{lemma}
The proof is deferred to Section \ref{weilsection}.

\end{subsection}

\begin{subsection}{Acknowledgements}
The author would like to thank his doctoral advisor, Dr. Matthew P. Young, for the invaluable help and advice he has provided throughout the course of writing this paper.
\end{subsection}
\end{section}

\begin{section}{Preliminaries and Notation}
\setcounter{subsection}{1}

\begin{itemize}
\item The notation $\cond(\chi)$ denotes the conductor of the Dirichlet character $\chi$, and the notation $\chi_0$ denotes the principal character to a particular modulus. Sometimes, if it is unambiguous, we will also use the symbol $1$ to denote the principal character, although it should be understood in all cases that the principal character vanishes at primes dividing the modulus.
\item The notation $e(z)$ is shorthand for $e^{2\pi iz}$. Similarly, the notation $e_q(z)$ is shorthand for $e^{\frac{2\pi iz}q}$.
\item For a family $\mathcal{F}$ of Hecke cuspforms of level $q$, define the sum of Fourier coefficients
\begin{equation}
\Delta_{\mathcal{F}}(m, n) = \sum_{f\in\mathcal{F}}\frac{\lambda_f(m)\overline{\lambda}_f(n)}{\left<f,f\right>_q},\label{deltadef}
\end{equation}
where $\lambda_f(n)$ is the $n$-th Hecke eigenvalue of $f$, and we have $\lambda_f(n) = 0$ if $n$ is not a positive integer; and $\left<f,g\right>_q$ is as defined below. For convenience, let $\Delta_{\mathcal{F}}(m) = \Delta_{\mathcal{F}}(m, 1)$. Also recall that the Hecke eigenvalues appear in the Fourier series
\begin{equation}
f(z) = \sum_{n\ge 1}\lambda_f(n)n^{\frac{\kappa-1}2}e(nz).
\end{equation}
\item The \textbf{Petersson inner product} for the congruence subgroup $\Gamma_0(q)$ is defined using the probability normalization as
\begin{equation}
\left<f, g\right>_q = \frac3{\pi\left[\Gamma_0(1): \Gamma_0(q)\right]}\iint_{\Gamma_0(q)\setminus\mathbb{H}}y^{\kappa}f(z)\overline{g(z)}\frac{dx\;dy}{y^2}.\label{petprod}
\end{equation}
\item The notation $\mathcal{S}_{\kappa}(q,\chi)$ denotes the vector space of cuspforms of weight $\kappa$, level $q$, and central character $\chi$. The notation $\mathcal{H}_{\kappa}(q, \chi)$ denotes the set of newforms in $\mathcal{S}_{\kappa}(q,\chi)$, which is an orthonormal set.
\item The \textbf{Petersson trace formula} states that, if $\mathcal{B}$ is an orthonormal basis for $\mathcal{S}_{\kappa}(q,\chi)$, then
\begin{equation}
\Delta_{\mathcal{B}}(m, n) = \frac{\pi(4\pi)^{\kappa-1}\left[\Gamma_0(1):\Gamma_0(q)\right]}{3\Gamma(\kappa-1)}\left[\delta(m,n) + 2\pi i^{-\kappa}\sum_{\substack{c>0\\c\equiv 0\pmod q}}c^{-1}S_\chi(m,n;c)J_{\kappa-1}\left(\frac{4\pi\sqrt{mn}}c\right)\right].\label{petersson}
\end{equation}
The notation used in this formula is clarified below.
\item $\Gamma_0(q)$ denotes the congruence subgroup of $\mathrm{SL}_2(\mathbb{Z})$ consisting of the matrices whose bottom left entry is divisible by $q$.
\item The symbol $\delta_{\text{condition}}$ is understood to equal 1 when the specified condition is true and 0 when it is false. The notation $\delta(m, n)$ is understood to mean $\delta_{m=n}$.
\item For integers $m, n$ and a positive integer $c$, the \textbf{ordinary Kloosterman sum} is defined as
\begin{equation}
S(m,n;c) = \sum_{\substack{x\pmod c \\ (x, c) = 1}}e_c(mx+n\overline{x}).
\end{equation}
Additionally, if $\chi$ is a Dirichlet character modulo $c$, the \textbf{twisted Kloosterman sum} is defined as
\begin{equation}
S_{\chi}(m,n;c) = \sum_{\substack{x\pmod c \\ (x, c) = 1}}\overline{\chi}(x)e_c(mx+n\overline{x}).
\end{equation}
\item Twisted Kloosterman sums have a series decomposition into multiplicative characters given by
\begin{equation}
S_{\chi}(m,1;c) = \frac1{\varphi(c)}\sum_{\psi\pmod c}\overline{\psi}(m)\tau(\psi)\tau(\chi\psi),\label{klfourier}
\end{equation}
valid when $(c, m) = 1$. This can be proved by opening up the Gauss sums and bringing the sum over $\psi$ to the inside.
\item Kloosterman sums satisfy the \textbf{Weil bound}:
\begin{equation}
|S(m,n;c)| \le c^{\frac12}(m,n,c)^{\frac12}\tau(c).\label{weil}
\end{equation}
As mentioned previously, the Weil bound does not hold in general for twisted Kloosterman sums. However, Theorem 9.3 of \cite{KL} implies the weaker result
\begin{equation}
|S_{\chi}(m,n;c)| \le c^{\frac12}q^{\frac12}(m,n,c)^{\frac12}\tau(c),\label{weilkl}
\end{equation}
if $q = \cond(\chi)$, and we also have the on-average result given in \eqref{weilavg}.
\begin{remark}
This result is occasionally useful to us, since we sometimes restrict $c$ to be much larger than $q$. However, it is worse than trivial when $c \asymp q$.
\end{remark}
\item For a Hecke cuspform $f$ and a Dirichlet character $\chi$, the notation $f_{\chi}$ denotes the twist $f\otimes \chi$; that is, the cuspform with Fourier series
\begin{equation}
f_{\chi}(z) = \sum_{n\ge 1}\lambda_f(n)\chi(n)n^{\frac{\kappa-1}2}e(nz).
\end{equation}
\item As introduced in \cite{ALE}, we say a modular form $f$ of level $q$ is an \textbf{oldform} if it satisfies $f(z) = g(dz)$, where $g$ is a cuspform of level  $q'$ properly dividing $q$, and $d$ divides $\frac{q}{q'}$. A modular form that is orthogonal (with respect to the Petersson inner product) to every oldform of level $q$ is called a \textbf{newform}.
\item If a Hecke cuspform $f$ of level $q$ and central character $\chi$ with Hecke eigenvalues $\lambda_f(n)$ is a newform, we define its associated $L$-function as
\begin{equation}
L(f, s) = \sum_{n=1}^{\infty}\frac{\lambda_f(n)}{n^s} = \prod_p\left(1 - \frac{\lambda_f(p)}{p^s} + \frac{\chi(p)}{p^{2s}}\right).
\end{equation}
This definition satisfies all the usual properties of an $L$-function if $f$ is a newform, but not if $f$ is an oldform -- in particular, for many oldforms we even have $\lambda_f(1) = 0$.
\item The notation $J_{\kappa-1}(x)$ denotes a $J$-Bessel function. These satisfy the estimates
\begin{align}
J_{\kappa-1}(x) &\ll x^{\kappa-1}, \label{jestk} \\
J_{\kappa-1}(x) &\ll x, \label{jestx} \\
J_{\kappa-1}(x) &\ll x^{-\frac12}, \label{jestsqrt}
\end{align}
where the implied constants are absolute in \eqref{jestk} and \eqref{jestx}, and the implied constant in \eqref{jestsqrt} depends only on $\kappa$.
\item The group density distributions, as given on page 409 of \cite{KS}, are
\begin{align}
W(\mathrm{U})(x) &= 1,\label{wux} \\
W(\mathrm{O})(x) &= 1 + \frac12\delta_0(x),\label{wox} \\
W(\mathrm{SO}(\mathrm{even}))(x) &= 1 + \frac{\sin(2\pi x)}{2\pi x},\label{wsoex} \\
W(\mathrm{SO}(\mathrm{odd}))(x) &= 1 - \frac{\sin(2\pi x)}{2\pi x} + \delta_0(x),\label{wsoox}
\end{align}
where $\delta_0(x)$ denotes the Dirac delta distribution.
\item For a positive integer $n$ and an $L$-function $L(f, s)$, the \textbf{generalized von Mangoldt function} $\Lambda_f(n)$ corresponding to $L(f, s)$ is given by the series expansion
\begin{equation}
-\frac{L'}{L}(f, s) = \sum_{n\ge 1}\frac{\Lambda_f(n)}{n^s}.\label{vmdef}
\end{equation}
\item The notation $||f||_{\infty}$ for a bounded function on the real line is understood to mean the supremum value of $|f|$ over all of $\mathbb{R}$.
\item The symbol $\varphi(n)$ denotes Euler's totient function (not to be confused with $\phi(x)$ which denotes a test function). It satisfies $\varphi(n) \le n$.
\item The symbol $\tau(n)$, for $n$ a positive integer, denotes the sum of the divisors of $n$. It satisfies $\tau(n) \ll n^{\varepsilon}$.
\item The symbol $\tau(\chi)$, for $\chi$ a Dirichlet character modulo $q$, denotes the Gauss sum
\begin{equation}
\tau(\chi) = \sum_{x\pmod q}\chi(x)e_q(x).
\end{equation}
It satisfies $|\tau(\chi)| = \sqrt{q}$ when $\chi$ is primitive to the modulus $q$.
\item The digamma function is denoted by $\Psi(x) = \frac{\Gamma'(x)}{\Gamma(x)}$.
\item The symbol $\vartheta_{\chi}(x)$ is the twisted Chebyshev function
\begin{equation}
\vartheta_{\chi}(x) = \sum_{p\le x}\chi(p)\log p.\label{chebyshev}
\end{equation}
\item For a function $F$, $\widehat{F}$ denotes its Fourier transform and $\widetilde{F}$ denotes its Mellin transform. Notation such as $\widehat{F}\thinspace'$, $\widehat{F}\thinspace''$, and $\widehat{F}\thinspace^{(n)}$ always indicates a derivative of the Fourier transform, as opposed to the Fourier transform of a derivative.
\item The symbol $\nu_p(n)$ denotes the largest nonnegative integer $k$ such that $p^k \mid n$.
\item The \textbf{floor square root} of a positive integer $n$, denoted $\flrt(n)$, is defined by
\begin{equation}
\flrt(n) = \prod_{p\mid n}p^{\left\lfloor\frac{\nu_p(n)}2\right\rfloor}.\label{flrtdef}
\end{equation}
\end{itemize}

\end{section}

\begin{section}{The Thinnest Family} \label{thinsection}

In this section, we prove Theorem \ref{theorem1}.

\begin{subsection}{The Explicit Formula}\label{explicitsct}
Our first goal will be to develop a Riemann-style explicit formula for the average one-level density of the thin family $\mathcal{F}_{\kappa}(q,\chi,\eta)$. This is given as follows:
\begin{lemma}
If $\mathcal{F} = \mathcal{F}_{\kappa}(q,\chi,\eta)$, $\mathcal{B} = \mathcal{H}_{\kappa}(q,\overline{\chi}^2)$, $R=\left(\frac{q}{2\pi}\right)^2$, $w$ is as in \eqref{petwt}, and $q$ is large enough to ensure that $\mathcal{B}$ is nonempty, then
\begin{multline}
\mathcal{D}_1(\mathcal{F},\phi,R,w) = \widehat{\phi}(0) + \frac2{\log R}\mathcal{I}(\kappa,\phi,R) \\
- \sum_{\ell}\sum_{e\ge 1}\frac{\left[\chi(\ell^e)\Delta_{\mathcal{B}}(\ell^e) - \chi(\ell^{e-2})\Delta_{\mathcal{B}}(\ell^{e-2})\right]\left[\eta(\ell^e) + \overline{\eta}(\ell^e)\right]\log\ell}{\ell^{e/2}\Delta_{\mathcal{H}}(1)\log R}\widehat{\phi}\left(\frac{e\log \ell}{\log R}\right),\label{oldavg1}
\end{multline}
where 
\begin{equation}
\mathcal{I}(\kappa,\phi,R) = \int_{-\infty}^{\infty}\Psi\left(\frac{\kappa}2+\frac{2\pi ix}{\log R}\right)\phi(x)\;dx,\label{idef}
\end{equation}
and $\Delta_{\mathcal{B}}(m)$ is as in \eqref{deltadef}, and the sum over $\ell$ runs over all primes.
\end{lemma}

\begin{remark}
If $\phi(x)$ is real, then $\mathcal{I}(\kappa,\phi,R)$ is real since the imaginary parts cancel out because $\phi$ is even. Moreover, for $a\ge 0$ we have the series expansion
\begin{equation}
\Re\left[\Psi(a+bi)\right] = \Psi(a) + b^2\sum_{k=a}^{\infty}\frac1{k(k^2+b^2)}.
\end{equation}
We can compare this sum to the integral
\begin{equation}
\int_a^{\infty}\frac{dx}{x(x^2+b^2)} = \frac1{2b^2}\log\left(1+\frac{b^2}{a^2}\right).
\end{equation}
Thus we have
\begin{equation}
\Re\left[\Psi\left(\frac{\kappa}2+\frac{2\pi ix}{\log R}\right)\right] \le \Psi\left(\frac{\kappa}2\right) + 1 + \frac2{\kappa^2},
\end{equation}
and so if $\phi(x)$ is real and positive, we have
\begin{equation}
\mathcal{I}(\kappa,\phi,R) \le \left[\Psi\left(\frac{\kappa}2\right) + 1 + \frac2{\kappa^2}\right]\widehat{\phi}(0).
\end{equation}
\end{remark}

\begin{remark}
In particular, this series expansion shows that $\Re\left[\Psi\left(\frac{\kappa}2+\frac{2\pi ix}{\log R}\right)\right] > 0$, so if $\phi(x)$ is real and positive, then $\mathcal{I}(\kappa,\phi,R)$ is both positive and bounded.
\end{remark}

\begin{remark}
We also may compute the asymptotic expansion
\begin{equation}
\mathcal{I}(\kappa,\phi,R) = \sum_{n=0}^{N-1}\frac{(-1)^n\pi^{2n}\Psi^{(2n)}\left(\frac{\kappa}2\right)\widehat{\phi}^{(2n)}(0)}{(2n)!}\left(\frac2{\log R}\right)^{2n} + O\left(\frac{\widehat{\phi}^{(2N)}(0)}{(\log R)^{2N}}\right),
\end{equation}
where the implied constant depends only on $N$.
\end{remark}

To prove \eqref{oldavg1}, we first expand out the one-level density for a single modular form.
\begin{lemma}
Let $f \in \mathcal{H}_{\kappa}(q,\chi)$. Then
\begin{equation}
D_1(f,\phi,R) = \frac{\log\left(\frac{q^2}{4\pi^2}\right)}{\log R}\widehat{\phi}(0) + \frac2{\log R}\mathcal{I}(\kappa,\phi,R) - \frac1{\log R}\sum_{n\ge 1}\frac{\Lambda_f(n) + \overline{\Lambda_f}(n)}{\sqrt n}\widehat{\phi}\left(\frac{\log n}{\log R}\right),\label{olds}
\end{equation}
where $\Lambda_f(n)$ is as in \eqref{vmdef}.
\end{lemma}

\begin{proof}
This is a direct application of Theorem 5.12 in \cite{IK}.
\end{proof}

Now we want to compute the generalized von Mangoldt function of a member of our family $\mathcal{F}_{\kappa}(q,\chi,\eta)$.
\begin{lemma}
Let $q$, $\chi$, and $\eta$ be as in \eqref{family1}. (In particular, $\chi\eta$ and $\overline{\chi}\eta$ have conductor $q$.) Let $f\in\mathcal{H}_{\kappa}(q, \overline{\chi}^2)$ so that $f_{\chi\eta} \in \mathcal{F}_{\kappa}(q,\chi,\eta)$. Then for a positive integer $n$,
\begin{equation}
\Lambda_{f_{\chi\eta}}(n) = \begin{cases}
\left[\chi(\ell^e)\lambda_f(\ell^e) - \chi(\ell^{e-2})\lambda_f(\ell^{e-2})\right]\eta(\ell^e)\log\ell & (n = \ell^e\text{ where $\ell\nmid q$ is prime and $e\ge 1$}) \\
0 & (\text{otherwise})
\end{cases},\label{vmexpr}
\end{equation}
where $\lambda_f(n)$ is the $n$-th Hecke eigenvalue of $f$, and we have $\lambda_f(n) = 0$ if $n$ is not a positive integer.
\end{lemma}

\begin{proof}
Due to (5.26) of \cite{IK}, we have
\begin{equation}
\Lambda_{f_{\chi\eta}}(n) = \begin{cases}
\left[\alpha_1(\ell)^e + \alpha_2(\ell)^e\right]\log\ell & (n = \ell^e\text{ where $\ell\nmid q$ is prime and $e\ge 1$}) \\
0 & (\text{otherwise})
\end{cases},
\end{equation}
where $\alpha_1, \alpha_2$ are the local roots of $L(s, f_{\chi\eta})$. To establish the result of \eqref{vmexpr}, we write the local factor at a prime $\ell\nmid q$ in the Euler product as
\begin{equation}
\sum_{e=0}^{\infty}\frac{\chi(\ell^e)\eta(\ell^e)\lambda_f(\ell^e)}{\ell^{es}} = \left(1-\frac{\alpha_1(\ell)}{\ell^s}\right)^{-1}\left(1-\frac{\alpha_2(\ell)}{\ell^s}\right)^{-1} = \sum_{e_1=0}^{\infty}\sum_{e_2=0}^{\infty}\frac{\alpha_1(\ell)^{e_1}\alpha_2(\ell)^{e_2}}{\ell^{s(e_1+e_2)}} = \sum_{e=0}^{\infty}\sum_{d=0}^e\frac{\alpha_1(\ell)^d\alpha_2(\ell)^{e-d}}{\ell^{es}},
\end{equation}
and match coefficients to get
\begin{equation}
\chi(\ell^e)\eta(\ell^e)\lambda_f(\ell^e) = \sum_{d=0}^e\alpha_1(\ell)^d\alpha_2(\ell)^{e-d} = \alpha_1(\ell)^e + \alpha_2(\ell)^e + \sum_{d=1}^{e-1}\alpha_1(\ell)^d\alpha_2(\ell)^{e-d}.
\end{equation}
Now
\begin{equation}
\sum_{d=1}^{e-1}\alpha_1(\ell)^d\alpha_2(\ell)^{e-d} = \alpha_1(\ell)\alpha_2(\ell)\sum_{d=0}^{e-2}\alpha_1(\ell)^d\alpha_2(\ell)^{e-2-d} = \chi(\ell^{e-2})\eta(\ell^e)\lambda_f(\ell^{e-2}),
\end{equation}
since $\alpha_1(\ell)\alpha_2(\ell) = \eta^2(\ell)$. Substituting this in gives the result.
\end{proof}

We can substitute this result into \eqref{olds}. Since $f_{\chi}$ lives in $\mathcal{H}_{\kappa}(q^2, 1)$ and is thus self-dual, $\chi(n)\lambda_f(n)$ and so we have
\begin{corollary}
Let $f\in\mathcal{H}_{\kappa}(q,\overline{\chi}^2)$ so that $f_{\chi\eta} \in \mathcal{F}_{\kappa}(q,\chi,\eta)$. Then if $R = \left(\frac{q}{2\pi}\right)^2$, we have
\begin{multline}
D_1(f_{\chi\eta},\phi,R) = \widehat{\phi}(0) + \frac2{\log R}\mathcal{I}(\kappa,\phi,R) \\
- \sum_{\ell}\sum_{e\ge 1}\frac{\left[\chi(\ell^e)\lambda_f(\ell^e) - \chi(\ell^{e-2})\lambda_f(\ell^{e-2})\right]\left[\eta(\ell^e) + \overline{\eta}(\ell^e)\right]\log\ell}{\ell^{e/2}\log R}\widehat{\phi}\left(\frac{e\log \ell}{\log R}\right).\label{oldind}
\end{multline}
\end{corollary}
\begin{remark}
It should be noted that this is really a finite sum, due to the compact support of $\widehat{\phi}$, so series rearrangements are freely allowed.
\end{remark}

Substituting this into \eqref{oldavgdef} gives the explicit formula \eqref{oldavg1}. We can split the remaining sum in \eqref{oldavg1} up into diagonal and non-diagonal terms, writing
\begin{equation}
\mathcal{D}_1(\mathcal{F},\phi,R,w) = \widehat{\phi}(0) + \frac2{\log R}\mathcal{I}(\kappa,\phi,R) + \frac2{\log R}\mathcal{S}(\phi,R,\eta) + \mathcal{E}(\mathcal{F},\phi, R),\label{diagsep}
\end{equation}
where
\begin{equation}
\mathcal{S}(\phi,R,\eta) = \sum_{\ell}\frac{\Re\left[\eta^2(\ell)\right]\log\ell}{\ell}\widehat{\phi}\left(\frac{2\log\ell}{\log R}\right),\label{sdef}
\end{equation}
and
\begin{equation}
\mathcal{E}(\mathcal{F},\phi,R) = -\sum_{\ell}\sum_{e\ge 1}\frac{\left[\chi(\ell^e)\Delta_{\mathcal{B}}(\ell^e) - \delta_{e\ge 3}\chi(\ell^{e-2})\Delta_{\mathcal{B}}(\ell^{e-2})\right]\left[\eta(\ell^e) + \overline{\eta}(\ell^e)\right]\log\ell}{\ell^{e/2}\Delta_{\mathcal{B}}(1)\log R}\widehat{\phi}\left(\frac{e\log \ell}{\log R}\right).\label{edef}
\end{equation}
We can also define the related expression
\begin{equation}
E_{\mathcal{B}}(m) = 2\pi i^{-\kappa}\sum_{\substack{c>0\\c\equiv 0\pmod q}}c^{-1}S_{\overline{\chi}^2}(m,1;c)J_{\kappa-1}\left(\frac{4\pi\sqrt{m}}c\right)
\end{equation}
to be the error term in the Petersson trace formula \eqref{petersson}, so that when $m\neq 1$, we have
\begin{equation}
\frac{\Delta_{\mathcal{B}}(m)}{\Delta_{\mathcal{B}}(1)} = \frac{E_{\mathcal{B}}(m)}{1 + E_{\mathcal{B}}(1)}.
\end{equation}
(Again, we are assuming that $q$ is large enough that $\mathcal{B}$ is nonempty, which implies that $\Delta_{\mathcal{B}}(1)$ is nonzero.) We also may write $\mathcal{E}(\mathcal{F},\phi,R)$ in terms of these error terms:
\begin{equation}
\mathcal{E}(\mathcal{F},\phi,R) = -\sum_{\ell}\sum_{e\ge 1}\frac{\left[\chi(\ell^e)E_{\mathcal{B}}(\ell^e) - \delta_{e\ge 3}\chi(\ell^{e-2})E_{\mathcal{B}}(\ell^{e-2})\right]\left[\eta(\ell^e) + \overline{\eta}(\ell^e)\right]\log\ell}{\ell^{e/2}(1 + E_{\mathcal{B}}(1))\log R}\widehat{\phi}\left(\frac{e\log \ell}{\log R}\right).
\end{equation}

In the following sections, we will compute the asymptotics of these terms separately, to extract the remaining main terms in \eqref{result1a} and \eqref{result1b}. We have
\begin{lemma}
If $\eta$ is quadratic, we have
\begin{equation}
\mathcal{S}(\phi,R,\eta) = \frac{\phi(0)\log R}4 + \widehat{\phi}(0) - \mathcal{J}(q, \phi, R),\label{squad}
\end{equation}
where
\begin{equation}
\mathcal{J}(q, \phi, R) = \int_1^{\infty}\Phi_R'(x)\left[\vartheta_{\chi_0}(x)-x\right] \;dx,\label{jdef}
\end{equation}
\begin{equation}
\Phi_R(x) = \frac1x\widehat{\phi}\left(\frac{2\log x}{\log R}\right),\label{phirdef}
\end{equation}
and $\vartheta_{\chi_0}(x)$ is the Chebyshev function \eqref{chebyshev} twisted by the principal character $\chi_0$ modulo $q$.
Otherwise,
\begin{equation}
\mathcal{S}(\phi,R,\eta) = -\mathcal{J}(q, \phi, R, \eta^2),\label{snquad}
\end{equation}
where
\begin{equation}
\mathcal{J}(q, \phi,R,\chi) = \int_1^{\infty}\Phi_R'(x)\Re\left[\vartheta_{\chi}(x)\right]\;dx.\label{jgendef}
\end{equation}
\label{slemma}
\end{lemma}
\begin{remark}
The $\frac{\phi(0)\log R}4$ in \eqref{squad}, after multiplying by $\frac2{\log R}$, combines with the $\widehat{\phi}(0)$ in the explicit formula \eqref{oldavg1} to give the integral $\int_{-\infty}^{\infty}\phi(x)W(\mathrm{O})(x)\;dx$ which reflects the orthogonal symmetry of the family.
\end{remark}
\begin{remark}
These lower-order main terms have asymptotic expansions as follows for an arbitrary positive integer $N$:
\begin{align}
\mathcal{J}(q, \phi, R) &= -\mathcal{R}_0\widehat{\phi}(0) + \sum_{n=1}^{N-1}\frac{\widehat{\phi}\thinspace^{(2n)}(0)}{(2n)!}\left(\frac2{\log R}\right)^{2n}\left(2n\mathcal{R}_{2n-1}-\mathcal{R}_{2n}\right) + O\left(\frac{||\widehat{\phi}\thinspace^{(2N)}||_{\infty}}{(\log R)^{2N}}\right), \\
\mathcal{J}(q, \phi,R,\chi) &= -\mathcal{R}_{0,\chi}\widehat{\phi}(0) + \sum_{n=1}^{N-1}\frac{\widehat{\phi}\thinspace^{(2n)}(0)}{(2n)!}\left(\frac2{\log R}\right)^{2n}\left(2n\mathcal{R}_{2n-1,\chi}-\mathcal{R}_{2n,\chi}\right) + O\left(\frac{||\widehat{\phi}\thinspace^{(2N)}||_{\infty}}{(\log R)^{2N}}\right),
\end{align}
where the implied constants depend only on $N$, and the $\mathcal{R}_j, \mathcal{R}_{j,\chi}$ are absolute constants given by
\begin{align}
\mathcal{R}_j &= \int_1^{\infty}\frac{(\log x)^j\left[\vartheta_{\chi_0}(x)-x\right]}{x^2}\;dx, \\
\mathcal{R}_{j,\chi} &= \int_1^{\infty}\frac{(\log x)^j\Re\left[\vartheta_{\chi}(x)\right]}{x^2}\;dx.
\end{align}
\end{remark}
We would also like to bound the non-diagonal terms $\mathcal{E}(\mathcal{F},\phi, R)$ as defined in \eqref{edef} and thus place them into an error term. The quality of the bound we will get is often better if we are allowed to assume the Generalized Riemann Hypothesis. Specifically, we have
\begin{lemma}
We have
\begin{equation}
\mathcal{E}(\mathcal{F},\phi, R) \ll ||\widehat{\phi}||_{\infty}\left(q^{\kappa(\theta-1)+\frac12+\varepsilon} + q^{-1+\varepsilon}\right),\label{unconderr}
\end{equation}
where the implied constant depends only on $\varepsilon$.\label{elemmauncond}
\end{lemma}
\begin{lemma}
Assuming GRH, we have
\begin{equation}
\mathcal{E}(\mathcal{F},\phi, R) \ll \left(||\widehat{\phi}||_{\infty} + ||\widehat{\phi}\thinspace''||_{\infty}\right)q^{\frac{\theta-1}2+\varepsilon},\label{conderr}
\end{equation}
where the implied constant depends only on $\kappa$ and $\varepsilon$.\label{elemmacond}
\end{lemma}
The remainder of this section consists of proving these three lemmas. Theorem \ref{theorem1} follows from their statements.

\end{subsection}

\begin{subsection}{Proof of Lemma \ref{slemma}}
First, we evaluate the diagonal $\mathcal{S}(\phi,R,\eta)$. This will allow us to extract the remaining main terms in \eqref{result1a} and \eqref{result1b}. Using summation by parts, we have
\begin{equation}
\mathcal{S}(\phi,R,\eta) = -\sum_{\ell}\Re\left[\eta^2(\ell)\right]\log\ell\int_{\ell}^{\infty}\Phi_R'(x)\;dx,
\end{equation}
where $\Phi_R(x)$ is as in \eqref{phirdef}. Since the integrand is compactly supported, we can rearrange to get
\begin{equation}
\mathcal{S}(\phi,R,\eta) = -\int_1^{\infty}\Phi_R'(x)\sum_{\ell\le x}\Re\left[\eta^2(\ell)\right]\log\ell\;dx = -\mathcal{J}(q, \phi,R,\eta^2),
\end{equation}
which gives the desired result for $\eta^2\neq 1$. However, if $\eta^2 = 1$, this term can have additional main terms extracted from it, since $\vartheta_{\chi_0}(x) \asymp x$. So we have
\begin{equation}
\mathcal{S}(\phi,R,\eta) = -\mathcal{J}(q, \phi,R,1) = -\int_1^{\infty}\Phi_R'(x)\vartheta_{\chi_0}(x)\;dx.
\end{equation}
Subtracting and adding the $x$ gives
\begin{equation}
\mathcal{S}(\phi,R,\eta) = -\int_1^{\infty}\Phi_R'(x)\left[\vartheta_{\chi_0}(x)-x\right]\;dx - \int_1^{\infty}x\Phi_R'(x)\;dx.
\end{equation}
The first integral just gives $\mathcal{J}(q, \phi, R)$ by our definition. Integrating the second by parts, we have
\begin{equation}
\int_1^{\infty}x\Phi_R'(x)\;dx = -\Phi_R(1) - \int_1^{\infty}\Phi_R(x)\;dx = -\widehat{\phi}(0) - \int_1^{\infty}\widehat{\phi}\left(\frac{2\log x}{\log R}\right)\frac{dx}x,
\end{equation}
and substituting $x = R^{y/2}$ gives
\begin{equation}
\int_1^{\infty}\widehat{\phi}\left(\frac{2\log x}{\log R}\right)\frac{dx}x = \frac{\log R}2\int_0^{\infty}\widehat{\phi}(y)\;dy = \frac{\phi(0)\log R}4.
\end{equation}
Substituting these in gives the result.

\end{subsection}

\begin{subsection}{Proof of Lemma \ref{elemmauncond}}

First, we bound the diagonal error term $E_{\mathcal{B}}(1)$:
\begin{lemma}
We have
\begin{equation}
E_{\mathcal{B}}(1) \ll q^{-1+\varepsilon},\label{ef1bound}
\end{equation}
where the implied constant depends only on $\varepsilon$.
\end{lemma}
\begin{proof}
We simply apply the estimates \eqref{weilkl} and \eqref{jestx}.
\end{proof}
\begin{remark}
We may compare this lemma to Corollary 14.24 in \cite{IK}; although for $\chi$ nonprincipal this result is not always correct since the Weil bound does not always apply to the Kloosterman sum $S_{\chi}(m,n;c)$. Weakening it by the necessary factor of $q^{\frac12}$ in \eqref{weilkl}, however, gives the result
\begin{equation}
E_{\mathcal{B}}(1) \ll q^{-1}\kappa^{-1}\tau(q),
\end{equation}
where the implied constant is absolute. This result always holds, and is slightly stronger than \eqref{ef1bound}.
\end{remark}
Next, for positive integers $e$ and $e'$, define
\begin{equation}
S_{\mathcal{F}}(e, e') = \sum_{\ell}\frac{\chi(\ell^{e'})E_{\mathcal{B}}(\ell^{e'})\left[\eta(\ell^e)+\overline{\eta}(\ell^e)\right]\log\ell}{\ell^{e/2}}\widehat{\phi}\left(\frac{e\log\ell}{\log R}\right),\label{sfdef}
\end{equation}
so that
\begin{equation}
\mathcal{E}(\mathcal{F},\phi,R) = -\frac1{\left(1 + E_{\mathcal{B}}(1)\right)\log R}\sum_{1\le e\le \theta\log_2(R)}\left[S_{\mathcal{F}}(e, e) - \delta_{e\ge 3}S_{\mathcal{F}}(e,e-2)\right].
\end{equation}
Opening up the sum $E_{\mathcal{B}}(\ell^{e'})$ gives
\begin{equation}
S_{\mathcal{F}}(e, e') = 2\pi i^{-\kappa}\sum_{\ell}\sum_{\substack{c>0 \\ c\equiv 0\pmod q}}\frac{\chi(\ell^{e'})S_{\overline{\chi}^2}(\ell^{e'},1;c)\left[\eta(\ell^e)+\overline{\eta}(\ell^e)\right]\log\ell}{c\ell^{e/2}}J_{\kappa-1}\left(\frac{4\pi\ell^{e/2}}c\right)\widehat{\phi}\left(\frac{e\log\ell}{\log R}\right).
\end{equation}
We will bound this using our on-average Weil bound for Kloosterman sums.
\begin{lemma}
If $\theta < 1$, we have
\begin{equation}
S_{\mathcal{F}}(e, e') \ll ||\widehat{\phi}||_{\infty}\left(q^{\theta\left(\kappa-2+\frac{2}e\right)-\left(\kappa-\frac12\right)+\varepsilon} + q^{-1+\varepsilon}\right),\label{boundsf}
\end{equation}
where the implied constant depends only on $\varepsilon$.
\end{lemma}
\begin{proof}
First, for a large positive real number $C$, define the tail as
\begin{equation}
S_{\mathcal{F}, C}(e, e') = 2\pi i^{-\kappa}\sum_{\ell}\sum_{\substack{c>C \\ c\equiv 0\pmod q}}\frac{\chi(\ell^{e'})S_{\overline{\chi}^2}(\ell^{e'},1;c)\left[\eta(\ell^e)+\overline{\eta}(\ell^e)\right]\log\ell}{c\ell^{e/2}}J_{\kappa-1}\left(\frac{4\pi\ell^{e/2}}c\right)\widehat{\phi}\left(\frac{e\log\ell}{\log R}\right).
\end{equation}
Using the estimates \eqref{weilkl} and \eqref{jestx}, we write this as
\begin{equation}
S_{\mathcal{F}, C}(e, e') \ll \sum_{\ell\le R^{\theta/e}}\sum_{\substack{c>C \\ c\equiv 0\pmod q}}c^{-\frac32+\varepsilon}q^{\frac12+\varepsilon}\ell^{\varepsilon}||\widehat{\phi}||_{\infty} \ll ||\widehat{\phi}||_{\infty}q^{2\theta-\frac12+\varepsilon}C^{-\frac12+\varepsilon}.
\end{equation}
Pick $C$ to be a fixed power of $q$ large enough that $S_{\mathcal{F},C}(e,e') \ll ||\widehat{\phi}||_{\infty}q^{-1+\varepsilon}$. (Something like $C = q^{100}$ should suffice.) Then using the compact support of $\widehat{\phi}$, we have
\begin{equation}
S_{\mathcal{F}}(e, e') \ll ||\widehat{\phi}||_{\infty}\sum_{\ell\le R^{\theta/e}}\sum_{\substack{0<c\le C \\ c\equiv 0\pmod q}}c^{-1}\ell^{-\frac{e}2+\varepsilon}\left|S_{\overline{\chi}^2}(\ell^{e'},1;c)\right|\left|J_{\kappa-1}\left(\frac{4\pi\ell^{e/2}}c\right)\right| + O\left(||\widehat{\phi}||_{\infty}q^{-1+\varepsilon}\right),
\end{equation}
where the implied constant depends only on $\varepsilon$. Next we use the estimate \eqref{jestk} to write
\begin{equation}
S_{\mathcal{F}}(e, e') \ll ||\widehat{\phi}||_{\infty}\sum_{\substack{0<c\le C \\ c\equiv 0\pmod q}}c^{-\kappa}\sum_{\ell\le R^{\theta/e}}\ell^{\frac{e(\kappa-2)}2+\varepsilon}\left|S_{\overline{\chi}^2}(\ell^{e'},1;c)\right| + O\left(||\widehat{\phi}||_{\infty}q^{-1+\varepsilon}\right).
\end{equation}
Now recalling the on-average Weil bound given in Lemma \ref{weilavglemma}, we establish the quick corollary
\begin{equation}
\sum_{1\le m\le M}m^{\alpha}|S_{\chi}(m,n;c)| \ll c^{\frac12+\varepsilon}M^{1+\alpha+\varepsilon}n^{1+\varepsilon} + c^{1+\varepsilon}M^{\alpha+\varepsilon}n^{\varepsilon},\label{weilcor}
\end{equation}
for $\alpha\ge 0$ and for $\chi$ a Dirichlet character modulo $c$. Using \eqref{weilcor}, we write
\begin{equation}
\sum_{\ell\le R^{\theta/e}}\ell^{\frac{e(\kappa-2)}2+\varepsilon}\left|S_{\overline{\chi}^2}(\ell^{e'},1;c)\right| \ll c^{\frac12+\varepsilon}R^{\theta\left(\frac{\kappa}2-1+\frac1e\right)+\varepsilon} + c^{1+\varepsilon}R^{\theta\left(\frac{\kappa}2-1\right)+\varepsilon}.
\end{equation}
Evaluating the outer sum out, we have
\begin{equation}
S_{\mathcal{F}}(e, e') \ll ||\widehat{\phi}||_{\infty}\left[q^{-\kappa+\frac12+\varepsilon}R^{\theta\left(\frac{\kappa}2-1+\frac1e\right)+\varepsilon} + q^{-\kappa+1+\varepsilon}R^{\theta\left(\frac{\kappa}2-1\right)+\varepsilon}C^{\varepsilon} + q^{-1+\varepsilon}\right].
\end{equation}
(The split at $C$ is necessary solely due to the case $\kappa=2$, since the second term gives the harmonic series which does not converge in this case.) Since $R = \left(\frac{q}{2\pi}\right)^2$ and $C$ is bounded by an absolute power of $q$, we have
\begin{equation}
S_{\mathcal{F}}(e, e') \ll ||\widehat{\phi}||_{\infty}\left[q^{\kappa(\theta-1)+\frac12-2\theta+\frac{2\theta}e+\varepsilon} + q^{\kappa(\theta-1)+1-2\theta+\varepsilon} + q^{-1+\varepsilon}\right].
\end{equation}
The result follows since the second term is $O(q^{-1+\varepsilon})$.
\end{proof}
Substituting this back in to the error term, and also applying \eqref{ef1bound} to bound the $1 + E_{\mathcal{B}}(1)$ in the denominator gives the result \eqref{unconderr}.

\end{subsection}

\begin{subsection}{Proof of Lemma \ref{elemmacond}}

First, we apply \eqref{ef1bound} and \eqref{boundsf} only to the terms $S_{\mathcal{F}}(e, e')$ where $e' = e-2$, and substitute $m = \ell^e$ in $S_{\mathcal{F}}(e,e)$ to get
\begin{equation}
\mathcal{E}(\mathcal{F},\phi,R) = -\frac1{\log R}\sum_{m\ge 1}\frac{\chi(m)\left[\eta(m)+\overline{\eta}(m)\right]\Lambda(m)E_{\mathcal{B}}(m)}{\sqrt{m}}\widehat{\phi}\left(\frac{\log m}{\log R}\right) + O\left(||\widehat{\phi}||_{\infty}q^{-\frac56+\varepsilon}\right).
\end{equation}
Consider just the tail of the sum in $E_{\mathcal{B}}(m)$, given by
\begin{equation}
\mathcal{E}_C(\mathcal{F},\phi,R) = -\frac{2\pi i^{-\kappa}}{\log R}\sum_{m\ge 1}\sum_{\substack{c > C \\ c\equiv 0\pmod q}}\frac{\chi(m)\left[\eta(m)+\overline{\eta}(m)\right]\Lambda(m)S_{\overline{\chi}^2}(m,1;c)}{c\sqrt{m}}J_{\kappa-1}\left(\frac{4\pi\sqrt{m}}c\right)\widehat{\phi}\left(\frac{\log m}{\log R}\right),
\end{equation}
for some large positive real number $C$. 
\begin{lemma}
We have
\begin{equation}
\mathcal{E}_C(\mathcal{F},\phi,R) \ll ||\widehat{\phi}||_{\infty}C^{-\frac12+\varepsilon}q^{2\theta-\frac12+\varepsilon},\label{ecbound}
\end{equation}
where the implied constant depends only on $\varepsilon$.
\end{lemma}
\begin{proof}
Consider the sum over $m$
\begin{equation}
\sum_{m\ge 1}\frac{\chi(m)\left[\eta(m)+\overline{\eta}(m)\right]\Lambda(m)S_{\overline{\chi}^2}(m,1;c)}{c\sqrt{m}}J_{\kappa-1}\left(\frac{4\pi\sqrt{m}}c\right)\widehat{\phi}\left(\frac{\log m}{\log R}\right).
\end{equation}
Splitting the prime power $m = \ell^e$ back up, and using the estimates \eqref{weilkl} and \eqref{jestx}, we write this as
\begin{equation}
\sum_{e\ge 1}\sum_{\ell\le R^{\theta/e}}\frac{\chi(\ell^e)\left[\eta(\ell^e)+\overline{\eta}(\ell^e)\right]S_{\overline{\chi}^2}(\ell^e,1;c)\log\ell}{c\ell^{e/2}}J_{\kappa-1}\left(\frac{4\pi\ell^{e/2}}c\right)\widehat{\phi}\left(\frac{e\log \ell}{\log R}\right) \ll \sum_{e\ge 1}\sum_{\ell\le R^{\theta/e}}c^{-\frac32+\varepsilon}q^{\frac12+\varepsilon}\ell^{\varepsilon}||\widehat{\phi}||_{\infty}.
\end{equation}
Substituting in $R = \left(\frac{q}{2\pi}\right)^2$, we have
\begin{equation}
\sum_{e\ge 1}\sum_{\ell\le R^{\theta/e}}\frac{\chi(\ell^e)\left[\eta(\ell^e)+\overline{\eta}(\ell^e)\right]S_{\overline{\chi}^2}(\ell^e,1;c)\log\ell}{c\ell^{e/2}}J_{\kappa-1}\left(\frac{4\pi\ell^{e/2}}c\right)\widehat{\phi}\left(\frac{e\log \ell}{\log R}\right) \ll ||\widehat{\phi}||_{\infty}c^{-\frac32+\varepsilon}q^{2\theta+\frac12+\varepsilon}.
\end{equation}
Summing this over $c$ gives the result.
\end{proof}

Also consider the terms where $(c, m) \neq 1$, given by
\begin{equation}
\mathcal{E}^*(\mathcal{F},\phi,R) = -\frac{2\pi i^{-\kappa}}{\log R}\sum_{m\ge 1}\sum_{\substack{c > 0 \\ c\equiv 0\pmod q \\ (c, m) \neq 1}}\frac{\chi(m)\left[\eta(m)+\overline{\eta}(m)\right]\Lambda(m)S_{\overline{\chi}^2}(m,1;c)}{c\sqrt{m}}J_{\kappa-1}\left(\frac{4\pi\sqrt{m}}c\right)\widehat{\phi}\left(\frac{\log m}{\log R}\right).
\end{equation}

\begin{lemma}
We have
\begin{equation}
\mathcal{E}^*(\mathcal{F},\phi,R) \ll ||\widehat{\phi}||_{\infty}q^{-1+\varepsilon},
\end{equation}
where the implied constant depends only on $\varepsilon$.
\end{lemma}

\begin{proof}
We write $m = \ell^e$ to get
\begin{equation}
\mathcal{E}^*(\mathcal{F},\phi,R) = -\frac{2\pi i^{-\kappa}}{\log R}\sum_{\substack{c > 0 \\ c\equiv 0\pmod q}}\sum_{\ell\mid c}\sum_{1\le e\le\theta \log_2(R)}\frac{\chi(\ell^e)\left[\eta(\ell^e)+\overline{\eta}(\ell^e)\right]S_{\overline{\chi}^2}(\ell^e,1;c)\log\ell}{c\ell^{e/2}}J_{\kappa-1}\left(\frac{4\pi\ell^{e/2}}c\right)\widehat{\phi}\left(\frac{e\log\ell}{\log R}\right).
\end{equation}
Bounding using \eqref{weilkl} and \eqref{jestx} gives
\begin{equation}
\mathcal{E}^*(\mathcal{F},\phi,R) \ll ||\widehat{\phi}||_{\infty}q^{\frac12+\varepsilon}\sum_{\substack{c > 0 \\ c\equiv 0\pmod q}}\sum_{\ell\mid c}\sum_{1\le e\le\theta \log_2(R)}c^{-\frac32+\varepsilon}\ell^{\varepsilon},
\end{equation}
and summing over $c$ gives the result.
\end{proof}

Now pick $C$ in \eqref{ecbound} large enough that that $\mathcal{E}_C(\mathcal{F},\phi,R) \ll ||\widehat{\phi}||_{\infty}q^{-\frac56+\varepsilon}$. For this $C$, using \eqref{klfourier} to expand out the Kloosterman sum into a Fourier series, we write
\begin{equation}
\mathcal{E}(\mathcal{F},\phi, R) = -\frac{2\pi i^{-\kappa}}{\log R}\sum_{\substack{0 < c \le C \\ c\equiv 0\pmod q}}\sum_{\vartheta\pmod c}\frac{\tau(\vartheta)\tau(\overline{\chi}^2\vartheta)\left[Q_{\chi\eta\overline{\vartheta}}(c) + Q_{\chi\overline{\eta\vartheta}}(c)\right]}{c\varphi(c)}  + O\left(||\widehat{\phi}||_{\infty}q^{-\frac56+\varepsilon}\right),\label{eqcsum}
\end{equation}
where
\begin{equation}
Q_{\chi}(c) = \sum_{m\ge 1}\frac{\chi(m)\Lambda(m)}{\sqrt{m}}J_{\kappa-1}\left(\frac{4\pi\sqrt{m}}c\right)\widehat{\phi}\left(\frac{\log m}{\log R}\right).\label{qcdef}
\end{equation}
Next we establish a bound for $Q_{\chi}(c)$ using GRH.
\begin{lemma}
Assume GRH. Then if $\chi$ is a nonprincipal character modulo $c$, and $\widehat{\phi}$ is supported on $(-\theta, \theta)$, then
\begin{equation}
Q_{\chi}(c) \ll ||\widehat{\phi}\thinspace''||_{\infty}R^{\frac{\theta}2+\varepsilon}c^{-1+\varepsilon},\label{qcbound}
\end{equation}
where the implied constant depends only on $\kappa$ and $\varepsilon$.
\end{lemma}

\begin{proof}
We use Mellin inversion. Let
\begin{equation}
F(x) = J_{\kappa-1}\left(\frac{4\pi\sqrt{x}}c\right)\widehat{\phi}\left(\frac{\log x}{\log R}\right).
\end{equation}
We write
\begin{equation}
Q_{\chi}(c) = \sum_{m\ge1}\frac{\chi(m)\Lambda(m)F(m)}{\sqrt{m}} = -\frac1{2\pi i}\int_{(\sigma)}\frac{L'}{L}\left(s+\frac12,\chi\right)\widetilde{F}(s)\;ds,
\end{equation}
valid for $\sigma > \frac12$. Now due to the Generalized Riemann Hypothesis, we have
\begin{equation}
\frac{L'}{L}\left(\frac12+\varepsilon+it,\chi\right) \ll (ct)^{\varepsilon},
\end{equation}
where the implied constant depends only on $\varepsilon$. Meanwhile, we define
\begin{equation}
F_1(x) = J_{\kappa-1}\left(\frac{4\pi\sqrt x}c\right) \qquad F_2(x) = \widehat{\phi}\left(\frac{\log x}{\log R}\right),
\end{equation}
so that $F(x) = F_1(x)F_2(x)$. Using formula (17.43.16) of \cite{GR}, we find the Mellin transform
\begin{equation}
\widetilde{F}_1(s) = \left(\frac{c}{2\pi}\right)^{2s}\frac{\Gamma\left(\frac{\kappa-1}2+s\right)}{\Gamma\left(\frac{\kappa+1}2-s\right)},
\end{equation}
and using a simple substitution, we find
\begin{equation}
\widetilde{F}_2(s) = \log(R)\phi\left(\frac{s\log R}{2\pi i}\right).
\end{equation}
Now due to the Mellin convolution theorem, we have
\begin{equation}
\widetilde{F}(s) = \frac1{2\pi i}\int_{(\sigma)}\widetilde{F}_1(u)\widetilde{F}_2(s-u)\;du,
\end{equation}
for $-\frac{\kappa-1}2 < \sigma < 1$. Thus, taking $\sigma = -\frac12+\varepsilon$, we have
\begin{equation}
Q_{\chi}(c) = \frac1{4\pi^2}\int_{(\varepsilon)}\int_{(-\frac12+\varepsilon)}\frac{L'}{L}\left(s+\frac12,\chi\right)\widetilde{F}_1(u)\widetilde{F}_2(s-u)\;du\;ds.
\end{equation}
Substituting $s=s+u$ gives
\begin{equation}
Q_{\chi}(c) = \frac1{4\pi^2}\int_{(\frac12)}\int_{(-\frac12+\varepsilon)}\frac{L'}{L}\left(s+u+\frac12,\chi\right)\widetilde{F}_1(u)\widetilde{F}_2(s)\;du\;ds.
\end{equation}
Writing $s = \frac12+it$ and $u = -\frac12+\varepsilon+iy$ gives
\begin{equation}
Q_{\chi}(c) = -\frac1{4\pi^2}\int_{-\infty}^{\infty}\int_{-\infty}^{\infty}\frac{L'}{L}\left(\frac12+\varepsilon+i(y+t), \chi\right)\widetilde{F}_1\left(-\frac12+\varepsilon+iy\right)\widetilde{F}_2\left(\frac12+it\right)\;dy\;dt.
\end{equation}
Now we write
\begin{equation}
\frac{L'}{L}\left(\frac12+\varepsilon+i(y+t), \chi\right) \ll c^{\varepsilon}(y+t)^{\varepsilon} \ll c^{\varepsilon}(1+|y|)^{\varepsilon}(1+|t|)^{\varepsilon},
\end{equation}
where the implied constant depends only on $\varepsilon$. Thus
\begin{equation}
Q_{\chi}(c) \ll c^{\varepsilon}\mathcal{J}_1(c)\mathcal{J}_2(c),
\end{equation}
where
\begin{align}
\mathcal{J}_1(c) &= \int_{-\infty}^{\infty}(1+|y|)^{\varepsilon}\left|\widetilde{F}_1\left(-\frac12+\varepsilon+iy\right)\right|\;dy, \\
\mathcal{J}_2(c) &= \int_{-\infty}^{\infty}(1+|t|)^{\varepsilon}\left|\widetilde{F}_2\left(\frac12+it\right)\right|\;dt.
\end{align}
For $\varepsilon$ sufficiently small, we claim these are bounded by
\begin{align}
\mathcal{J}_1(c) &\ll c^{-1+2\varepsilon},\label{j1bound} \\
\mathcal{J}_2(c) &\ll R^{\frac{\theta}2+\varepsilon}||\widehat{\phi}\thinspace''||_{\infty},\label{j2bound}
\end{align}
where the implied constants depend only on $\kappa$ and $\varepsilon$.
To prove \eqref{j1bound}, we estimate $\mathcal{J}_1(c)$ using Stirling's approximation as
\begin{equation}
\widetilde{F}_1\left(-\frac12+\varepsilon+iy\right) \asymp c^{-1+2\varepsilon}(1+|y|)^{-2+2\varepsilon}.
\end{equation}
For \eqref{j2bound}, we write
\begin{equation}
\mathcal{J}_2(c) = \log(R)\int_{-\infty}^{\infty}\left(1+|t|\right)^{\varepsilon}\left|\phi\left(\frac{(\frac12+it)\log R}{2\pi i}\right)\right|\;dt.
\end{equation}
Using the inverse Fourier transform, we write
\begin{equation}
\phi\left(\frac{(\frac12+it)\log R}{2\pi i}\right) = \int_{-\infty}^{\infty}e\left(\frac{x(\frac12+it)\log R}{2\pi i}\right)\widehat{\phi}(x)\;dx = \int_{-\infty}^{\infty}R^{x\left(\frac12+it\right)}\widehat{\phi}(x)\;dx.
\end{equation}
Integrating by parts twice gives this as
\begin{equation}
\phi\left(\frac{(\frac12+it)\log R}{2\pi i}\right) = \frac1{(\frac12+it)^2\log^2(R)}\int_{-\infty}^{\infty}R^{x\left(\frac12+it\right)}\widehat{\phi}\thinspace''(x)\;dx.
\end{equation}
This is bounded as
\begin{equation}
\left|\phi\left(\frac{(\frac12+it)\log R}{2\pi i}\right)\right| \le \frac1{|\frac12+it|^2\log^2(R)}\int_{-\theta}^{\theta}R^{\frac{x}2}||\widehat{\phi}\thinspace''||_{\infty}\;dx \le \frac{R^{\frac{\theta}2}||\widehat{\phi}\thinspace''||_{\infty}}{|\frac12+it|^2|\varepsilon-\sigma|\log^3(R)}.
\end{equation}
So we have
\begin{equation}
\phi\left(\frac{(\frac12+it)\log R}{2\pi i}\right) \ll \frac{R^{\frac{\theta}2+\varepsilon}||\widehat{\phi}\thinspace''||_{\infty}}{(\frac12+|t|)^2},
\end{equation}
and thus the result \eqref{j2bound} follows. Multiplying these together gives the final result.
\end{proof}
Meanwhile, for principal characters, the calculation is exactly the same as above, except there is a pole at $s = \frac12$ that the contour runs over, giving
\begin{equation}
Q_{\chi_0}(c) = \widetilde{F}\left(\frac12\right) - \frac1{2\pi i}\int_{(\sigma)}\frac{L'}{L}\left(s+\frac12,\chi\right)\widetilde{F}(s)\;ds,
\end{equation}
and thus
\begin{corollary}
For $\chi$ equal to the principal character $\chi_0$ modulo $c$, we have
\begin{equation}
Q_{\chi_0}(c) = \int_0^{\infty}\frac1{\sqrt x}J_{\kappa-1}\left(\frac{4\pi\sqrt{x}}c\right)\widehat{\phi}\left(\frac{\log x}{\log R}\right)\;dx + O\left(||\widehat{\phi}\thinspace''||_{\infty}R^{\frac{\theta}2+\varepsilon}c^{-1+\varepsilon}\right),\label{qcpbound}
\end{equation}
where the implied constants depends only on $\kappa$ and $\varepsilon$.
\end{corollary}

Using \eqref{qcbound}, we have
\begin{equation}
\frac{\tau(\vartheta)\tau(\overline{\chi}^2\vartheta)\left[Q_{\chi\eta\overline{\vartheta}}(c) + Q_{\chi\overline{\eta\vartheta}}(c)\right]}{c\varphi(c)} \ll ||\widehat{\phi}\thinspace''||_{\infty}R^{\frac{\theta}2+\varepsilon}c^{-2+\varepsilon},
\end{equation}
unless $\vartheta$ is either $\chi\eta$ or $\chi\overline{\eta}$. Summing over all $\vartheta$ and using \eqref{qcpbound} thus gives
\begin{multline}
\sum_{\vartheta\pmod c}\frac{\tau(\vartheta)\tau(\overline{\chi}^2\vartheta)\left[Q_{\chi\eta\overline{\vartheta}}(c) + Q_{\chi\overline{\eta\vartheta}}(c)\right]}{c\varphi(c)} \\
= \frac{\tau(\chi\eta)\tau(\overline{\chi}\eta) + \tau(\chi\overline{\eta})\tau(\overline{\chi\eta})}{c\varphi(c)}\int_0^{\infty}\frac1{\sqrt x}J_{\kappa-1}\left(\frac{4\pi\sqrt{x}}c\right)\widehat{\phi}\left(\frac{\log x}{\log R}\right)\;dx + O\left(||\widehat{\phi}\thinspace''||_{\infty}R^{\frac{\theta}2+\varepsilon}c^{-1+\varepsilon}\right),
\end{multline}
where the Gauss sums are taken modulo $c$. Next we sum over $c$ as well to get
\begin{multline}
\sum_{\substack{0<c\le C \\ c\equiv0\pmod q}}\sum_{\vartheta\pmod c}\frac{\tau(\vartheta)\tau(\overline{\chi}^2\vartheta)\left[Q_{\chi\eta\overline{\vartheta}}(c) + Q_{\chi\overline{\eta\vartheta}}(c)\right]}{c\varphi(c)} \\
= 2\tau(\chi\eta)\tau(\overline{\chi}\eta)\int_0^{\infty}\frac1{\sqrt x}\widehat{\phi}\left(\frac{\log x}{\log R}\right)\sum_{\substack{0<c\le C \\ c\equiv 0\pmod q}}\frac1{c\varphi(c)}\eta^2\left(\frac{c}q\right)\mu\left(\frac{c}q\right)J_{\kappa-1}\left(\frac{4\pi\sqrt{x}}c\right)\;dx + O\left(||\widehat{\phi}\thinspace''||_{\infty}R^{\frac{\theta}2+\varepsilon}q^{-1+\varepsilon}C^{\varepsilon}\right),
\end{multline}
using Lemma 3.1 of \cite{IK}. Now since $\chi\eta$ and $\overline{\chi}\eta$ are primitive modulo $q$, we have $|\tau(\chi\eta)\tau(\overline{\chi}\eta)| = q$. The sum over $c$ inside the integral is bounded using the triangle inequality and the estimate \eqref{jestsqrt} to get
\begin{equation}
\sum_{\substack{0<c\le C \\ c\equiv 0\pmod q}}\frac1{c\varphi(c)}\eta^2\left(\frac{c}q\right)\mu\left(\frac{c}q\right)J_{\kappa-1}\left(\frac{4\pi\sqrt{x}}c\right) \ll \sum_{\substack{c>0 \\ c\equiv 0\pmod q}}x^{-\frac14}c^{-\frac32} \ll x^{-\frac14}q^{-\frac32},
\end{equation}
where the implied constant depends only on $\kappa$.
Thus we have
\begin{equation}
2\tau(\chi\eta)\tau(\overline{\chi}\eta)\int_0^{\infty}\frac1{\sqrt x}\widehat{\phi}\left(\frac{\log x}{\log R}\right)\sum_{\substack{0<c\le C \\ c\equiv 0\pmod q}}\eta^2\left(\frac{c}q\right)\mu\left(\frac{c}q\right)J_{\kappa-1}\left(\frac{4\pi\sqrt{x}}c\right)\;dx \ll ||\widehat{\phi}||_{\infty}q^{-\frac12}\int_0^{R^{\theta}}x^{-\frac34}\;dx.
\end{equation}
So
\begin{equation}
\sum_{\substack{0<c\le C \\ c\equiv0\pmod q}}\sum_{\vartheta\pmod c}\frac{\tau(\vartheta)\tau(\overline{\chi}^2\vartheta)\left[Q_{\chi\eta\overline{\vartheta}}(c) + Q_{\chi\overline{\eta\vartheta}}(c)\right]}{c\varphi(c)} \ll ||\widehat{\phi}||_{\infty}q^{\frac{\theta-1}2+\varepsilon} + ||\widehat{\phi}\thinspace''||_{\infty}q^{\theta-1+\varepsilon}C^{\varepsilon},
\end{equation}
and so
\begin{equation}
\mathcal{E}(\mathcal{F},\phi,R) \ll ||\widehat{\phi}||_{\infty}q^{\frac{\theta-1}2+\varepsilon} + ||\widehat{\phi}\thinspace''||_{\infty}q^{\theta-1+\varepsilon}C^{\varepsilon}.
\end{equation}
Taking $C$ to be a fixed large power of $q$ (such as $q^{100}$) gives \eqref{conderr}.

\end{subsection}
\end{section}

\begin{section}{The Coset Family} \label{cosetsection}

In this section, we prove Theorem \ref{theorem2}.

\begin{subsection}{The Coset Sum}\label{cosetsumsection}

Now we seek to establish an analogous result to Theorem \ref{theorem1} for the average one-level density of the coset family. First, we can calculate a result analogous to the explicit formula \eqref{oldavg1} using the work of Section \ref{explicitsct}.
\begin{lemma}
If $q = p^k$, $\mathcal{F} = \mathcal{F}_{\kappa,\epsilon}(q,\chi,\eta)$, $\mathcal{B} = \mathcal{H}_{\kappa}(q,\overline{\chi}^2)$, $R=\left(\frac{q}{2\pi}\right)^2$, $w$ is as in \eqref{petwt}, and $q$ is large enough to ensure that $\mathcal{B}$ is nonempty, then
\begin{multline}
\mathcal{D}_1(\mathcal{F},\phi,R,w) = \widehat{\phi}(0) + \frac2{\log R}\mathcal{I}(\kappa,\phi,R) \\
-\frac2{\sigma_{\psi,\epsilon}(1, 1)\log R}\sum_{\ell\neq p}\sum_{e\ge 1}\frac{\left[\sigma_{\psi,\epsilon}\left(\ell^e,1\right) - \sigma_{\psi,\epsilon}\left(\ell^{e-2},1\right)\right]\log \ell}{\ell^{e/2}}\widehat{\phi}\left(\frac{e\log \ell}{\log R}\right),\label{oldavg2}
\end{multline}
where the sum over $\ell$ runs over all primes not equal to $p$, and $\sigma_{\psi,\epsilon}(m, n)$ is defined for positive integers $m, n$ as
\begin{equation}
\sigma_{\psi,\epsilon}(m, n) = \sum_{\chi\in\psi \widehat{G}_{p^j}}\left[1+\epsilon\chi(-1)\right]\chi(m)\overline{\chi(n)}\Delta_{\mathcal{B}}(m, n).
\end{equation}
(We define $\sigma_{\psi,\epsilon}(m,n) = 0$ whenever $m, n$ are not both positive integers.)
\end{lemma}

\begin{proof}
By the formula \eqref{oldavgdef} for one-level density, we have
\begin{equation}
\mathcal{D}_1(\mathcal{F},\phi,R,w) = \frac{\sum_{\chi\in \psi \widehat{G}_{p^j}}\sum_{f\in\mathcal{H}_{\kappa}\left(q,\overline{\chi}^2\right)}\frac{\left[1+\epsilon\chi(-1)\right]D_1(f_{\chi}, \phi, R)}{\left<f, f\right>_q}}{\sum_{\chi\in \psi \widehat{G}_{p^j}}\sum_{f\in\mathcal{H}_{\kappa}\left(q,\overline{\chi}^2\right)}\frac{1+\epsilon\chi(-1)}{\left<f, f\right>_q}}.
\end{equation}
The result follows from substituting \eqref{oldind} into this.
\end{proof}

\begin{lemma}
If $m, n$ are positive integers, we have
\begin{equation}
\sigma_{\psi,\epsilon}(m, n) = \frac{\pi(4\pi)^{\kappa-1}\left[\Gamma_0(1):\Gamma_0(q)\right]}{3\Gamma(\kappa-1)}\left[\delta(m, n)\varphi(p^j) + 2\pi i^{-\kappa}\sum_{\substack{c > 0 \\ c\equiv 0\pmod q}}c^{-1}\sigma_{\psi,\epsilon}(m,n;c)J_{\kappa-1}\left(\frac{4\pi\sqrt{mn}}c\right)\right],\label{sigmamn}
\end{equation}
where the coset sum
\begin{equation}
\sigma_{\psi,\epsilon}(m,n;c) = \sum_{\chi\in\psi \widehat{G}_{p^j}}\chi(m)\overline{\chi}(n)\left[1+\epsilon\chi(-1)\right]S_{\overline{\chi}^2}(m,n;c)\label{cosetsum}
\end{equation}
is defined for positive integers $m, n, c$ with $c$ divisible by $q$.\label{cosetsumlemma}
\end{lemma}

\begin{proof}
Applying the Petersson trace formula \eqref{petersson} to \eqref{oldavg2}, we write
\begin{multline}
\sigma_{\psi,\epsilon}(m, n) = \\
\frac{\pi(4\pi)^{\kappa-1}\left[\Gamma_0(1):\Gamma_0(q)\right]}{3\Gamma(\kappa-1)}\sum_{\chi\in\psi \widehat{G}_{p^j}}\left[1+\epsilon\chi(-1)\right]\chi(m)\overline{\chi}(n)\left[\delta(m, n) + 2\pi i^{-\kappa}\sum_{\substack{c > 0 \\ c\equiv 0\pmod q}}c^{-1}S_{\overline{\chi}^2}(m,n;c)J_{\kappa-1}\left(\frac{4\pi\sqrt{mn}}c\right)\right].
\end{multline}
Bringing the sum over $\chi$ to the inside gives the result.
\end{proof}

We would like to think of this sum as being analogous to a Kloosterman sum. In fact, we can establish a series of results based on how many extra copies of the prime $p$ divide $c$. First, we show that if it is divisible by $k$ extra copies of $p$, it is actually just an ordinary Kloosterman sum.

\begin{lemma}
If $b,m,n$ are positive integers not divisible by the prime $p$, $\chi$ is a Dirichlet character modulo $p^k$, and $c = bp^r$ where $r \ge 2k$, then
\begin{equation}
\chi(m)\overline{\chi}(n)S_{\overline{\chi}^2}(m,n;c) = S(m,n;c).\label{kloosterman2k}
\end{equation}
\end{lemma}

\begin{proof}
First, we factor to get
\begin{equation}
S_{\overline{\chi}^2}(m,n;c) = S(\overline{p}^rm,\overline{p}^rn;b)S_{\overline{\chi}^2}(\overline{b}m,\overline{b}n;p^r).
\end{equation}
Opening up the second Kloosterman sum gives
\begin{equation}
S_{\overline{\chi}^2}(\overline{b}m,\overline{b}n;p^r) = \sum_{\substack{x\pmod{p^r} \\ p\nmid x}}\chi^2(x)e_{p^r}(\overline{b}mx+\overline{b}n\overline{x}).
\end{equation}
Split $x = p^{r-k}x_1 + x_2$, where $x_1$ runs modulo $p^k$ and $x_2$ runs modulo $p^{r-k}$. Then the sum becomes
\begin{equation}
S_{\overline{\chi}^2}(\overline{b}m,\overline{b}n;p^r) = \sum_{x_1\pmod{p^k}}\sum_{\substack{x_2\pmod{p^{r-k}} \\ p\nmid x_2}}\chi^2(x_2)e_{p^r}\left(\overline{b}m(p^{r-k}x_1+x_2) + \overline{b}n\overline{(p^{r-k}x_1+x_2)}\right).
\end{equation}
Since $r \le 2(r-k)$, we have $\overline{p^{r-k}x_1+x_2} = -p^{r-k}x_1\overline{x_2}^2+\overline{x_2}$. We substitute this in and bring the sum over $x_1$ to the inside to write
\begin{equation}
S_{\overline{\chi}^2}(\overline{b}m,\overline{b}n;p^r) = \sum_{\substack{x_2\pmod{p^{r-k}} \\ p\nmid x_2}}\chi^2(x_2)e_{p^r}(\overline{b}mx_2+\overline{b}n\overline{x_2})\sum_{x_1\pmod{p^k}}e_{p^r}(\overline{b}x_1(m-n\overline{x_2}^2)).
\end{equation}
The inner sum vanishes unless $mx_2^2 \equiv n\pmod{p^k}$. Thus we have
\begin{equation}
S_{\overline{\chi}^2}(\overline{b}m,\overline{b}n;p^r) = p^k\sum_{\substack{x\pmod{p^{r-k}} \\ mx^2\equiv n\pmod{p^k}}}\chi^2(x)e_{p^r}(\overline{b}mx+\overline{b}n\overline{x}).
\end{equation}
Since $x^2 \equiv \overline{m}n\pmod{p^k}$, we have $\chi^2(x) = \overline{\chi}(m)\chi(n)$ for all $x$. Thus
\begin{equation}
\chi(m)\overline{\chi}(n)S_{\overline{\chi}^2}(\overline{b}m,\overline{b}n;p^r) = p^k\sum_{\substack{x\pmod{p^{r-k}} \\ mx^2\equiv n\pmod{p^k}}}e_{p^r}(\overline{b}mx+\overline{b}n\overline{x}).
\end{equation}
By our construction, this formula is independent of the choice of lifts from $p^{r-k}$ to $p^r$. Therefore, we write
\begin{equation}
\chi(m)\overline{\chi}(n)S_{\overline{\chi}^2}(\overline{b}m,\overline{b}n;p^r) = \sum_{\substack{x\pmod{p^r} \\ mx^2\equiv n\pmod{p^k}}}e_{p^r}(\overline{b}mx+\overline{b}n\overline{x}).
\end{equation}
This holds for any character $\chi$ modulo $p^k$, since $\chi$ does not appear on the right-hand side. Taking $\chi$ to be the trivial character gives the result.
\end{proof}

Substituting this into \eqref{cosetsum} gives
\begin{corollary}
If $b,m,n$ are positive integers not divisible by $p$, and $c = bp^r$ where $r \ge 2k$, then
\begin{equation}
\sigma_{\psi,\epsilon}(m,n;c) = \varphi(p^j)S(m,n;c).\label{cosetsum2k}
\end{equation}
\end{corollary}

Thus in the following, we are mainly concerned with the case where $r < 2k$. In general, we can split our coset sum up into a same-sign and an opposite-sign character sum.

\begin{lemma}
If $b,m,n$ are positive integers not divisible by $p$, and $c = bp^r$ where $r\ge k$, then
\begin{equation}
\sigma_{\psi,\epsilon}(m,n;c) = \psi(m)\overline{\psi}(n)\varphi(p^j)S(\overline{p}^rm,\overline{p}^rn;b)\left[\mathcal{K}_{\psi}^+(\overline{b}m,\overline{b}n;p^r) + \epsilon\psi(-1)\mathcal{K}_{\psi}^-(\overline{b}m,\overline{b}n;p^r)\right],\label{sigmasplit}
\end{equation}
where
\begin{equation}
\mathcal{K}_{\psi}^{\pm}(m,n;p^r) = \sum_{\substack{x\pmod{p^r} \\ mx^2\equiv \pm n\pmod{p^j}}}\psi^2(x)e_{p^r}(mx+n\overline{x}),\label{kpsidef}
\end{equation}
and $\overline{p}$ denotes the inverse of $p$ modulo $b$, and $\overline{b}$ denotes the inverse of $b$ modulo $p^r$.
\end{lemma}

\begin{proof}
Opening the Kloosterman sum gives
\begin{equation}
\sigma_{\psi,\epsilon}(m,n;c) = \sum_{\chi\in\psi \widehat{G}_{p^j}}\chi(m)\overline{\chi}(n)\left[1+\epsilon\chi(-1)\right]\sum_{\substack{x\pmod c \\ (x,c)=1}}\chi^2(x)e_c(mx+n\overline{x}).
\end{equation}
We switch the order of summation to get
\begin{equation}
\sigma_{\psi,\epsilon}(m,n;c) = \sum_{\substack{x\pmod c \\ (x,c)=1}}e_c(mx+n\overline{x})\sum_{\chi\in\psi \widehat{G}_{p^j}}\chi(mx^2)\overline{\chi}(n)\left[1+\epsilon\chi(-1)\right].
\end{equation}
Writing the character $\chi$ instead as $\psi\chi$ for some $\chi \in H_j$, we have 
\begin{equation}
\sigma_{\psi,\epsilon}(m,n;c) = \sum_{\substack{x\pmod c \\ (x,c)=1}}\psi(mx^2)\overline{\psi}(n)e_c(mx+n\overline{x})\sum_{\chi\pmod{p^k}}\chi(mx^2)\overline{\chi}(n)\left[1+\epsilon\psi(-1)\chi(-1)\right].
\end{equation}
Applying the orthogonality principle, we write
\begin{equation}
\sigma_{\psi,\epsilon}(m,n;c) = \varphi(p^j)\sum_{\substack{x\pmod c \\ (x,c)=1 \\ mx^2\equiv n\pmod{p^j}}}\psi(mx^2)\overline{\psi}(n)e_c(mx+n\overline{x}) + \epsilon\varphi(p^j)\sum_{\substack{x\pmod c \\ (x,c)=1 \\ mx^2\equiv -n\pmod{p^j}}}\psi(mx^2)\overline{\psi}(-n)e_c(mx+n\overline{x}).
\end{equation}
Factoring out $\psi(m)\overline{\psi}(n)\varphi(p^j)$, and factoring out the Kloosterman sum modulo $b$ as well, gives the result.
\end{proof}

Next we establish a pair of results showing that at most one of the $\mathcal{K}_{\psi}^{\pm}(m,n;p^r)$ is nonzero at a time.

\begin{lemma}
If $m,n$ are positive integers not divisible by $p$, and $k \le r < j+k$, then $\mathcal{K}_{\psi}^+(m,n;p^r) = 0$.\label{lemmakplus}
\end{lemma}

\begin{proof}
If $r \ge 2j$, we replace $x$ by $(1+p^{r-j})x$ in \eqref{kpsidef}, so that $\overline{x}$ is replaced by $(1-p^{r-j})\overline{x}$. Then the exponential is multiplied by $e_{p^j}(mx-n\overline{x})$, which is 1 since $mx^2\equiv n\pmod{p^j}$ so $mx\equiv n\overline{x}\pmod p$. However, the character is multiplied by $\psi^2(1+p^{r-j})$. This is not equal to 1, since $r-j < k$ and $\cond(\psi^2) = p^k$. So the sum equals $\psi^2(1+p^{r-j})\neq 1$ times itself, and thus vanishes. \\

If instead $r < 2j$, we replace $x$ instead by $(1+p^j)x$, so that $\overline{x}$ is replaced by $(1-p^j)\overline{x}$. Then the exponential is multiplied by $e_{p^{r-j}}(mx-n\overline{x})$, which is 1 as before since $r-j < j$; and the character is multiplied by $\psi^2(1+p^j)$, which is not 1 since $j < k$.
\end{proof}

\begin{lemma}
If $m,n$ are positive integers not divisible by $p$, and $r>k$, then $\mathcal{K}_{\psi}^-(m,n;p^r) = 0$. \label{lemmakminus}
\end{lemma}

\begin{proof}
If $r < 2k$, split $x = p^kx_1 + x_2$, where $x_1$ runs modulo $p^{r-k}$ and $x_2$ runs modulo $p^k$. Then
\begin{equation}
\mathcal{K}_{\psi}^-(m,n;p^r) = \sum_{x_1\pmod{p^{r-k}}}\sum_{\substack{x_2\pmod{p^k} \\ mx_2^2\equiv -n\pmod{p^j}}}\psi^2(x_2)e_{p^r}\left(m(p^kx_1+x_2)+n\overline{(p^kx_1+x_2)}\right).
\end{equation}
Rearranging and writing $\overline{p^kx_1+x_2} = -p^kx_1\overline{x_2}^2+\overline{x_2}$ gives
\begin{equation}
\mathcal{K}_{\psi}^-(m,n;p^r) = \sum_{\substack{x_2\pmod{p^k} \\ mx_2^2\equiv -n\pmod{p^j}}}\psi^2(x_2)e_{p^r}(mx_2+n\overline{x_2})\sum_{x_1\pmod{p^{r-k}}}e_{p^{r-k}}\left(x_1(m-n\overline{x_2}^2)\right).
\end{equation}
The inner sum vanishes unless $mx_2^2 \equiv n\pmod{p^{r-k}}$. But since $p\nmid n$ and $r > k$, we cannot have both $mx_2^2 \equiv -n\pmod p$ and $mx^2\equiv n\pmod p$, so every term in the sum over $x_2$ vanishes. \\

If instead $r\ge 2k$, this result is a direct consequence of \eqref{cosetsum2k}.
\end{proof}

As a result of this, we have a window of values of $r$ where the entire coset sum vanishes:
\begin{corollary}
If $b,m,n$ are positive integers not divisible by $p$, and $c = bp^r$ where $k<r<j+k$, then $\sigma_{\psi,\epsilon}(m,n;c) = 0$.\label{corljk}
\end{corollary}

\begin{remark}
In Section \ref{analogysection}, we will study a sum in the weight aspect which has a decomposition that mirrors the results of Lemma \ref{lemmakplus}, Lemma \ref{lemmakminus}, and Corollary \ref{corljk}.
\end{remark}

If the same-sign sum $\mathcal{K}_{\psi}^+(m,n;p^r)$ does not vanish, then it is exactly equal to a twisted Kloosterman sum:
\begin{lemma}
If $m,n$ are positive integers not divisible by $p$, and $c = bp^r$ where $r\ge j+k$, then
\begin{equation}
\mathcal{K}_{\psi}^+(m,n;p^r) = S_{\overline{\psi}^2}(m,n;p^r).\label{ktwisted}
\end{equation}\label{lemmaktwisted}
\end{lemma}

\begin{proof}
If $r < 2k$, split $x = p^kx_1+x_2$ as in Lemma \ref{lemmakminus} to get
\begin{equation}
\mathcal{K}_{\psi}^+(m,n;p^r) = \sum_{\substack{x_2\pmod{p^k} \\ mx_2^2\equiv n\pmod{p^j}}}\psi^2(x_2)e_{p^r}(mx_2+n\overline{x_2})\sum_{x_1\pmod{p^{r-k}}}e_{p^{r-k}}\left(x_1(m-n\overline{x_2}^2)\right).
\end{equation}
The inner sum vanishes unless $mx_2^2\equiv n\pmod{p^{r-k}}$. Since $r-k \ge j$, we just write this as
\begin{equation}
\mathcal{K}_{\psi}^+(m,n;p^r) = p^{r-k}\sum_{\substack{x_2\pmod{p^k} \\ mx_2^2\equiv n\pmod{p^{r-k}}}}\psi^2(x_2)e_{p^r}(mx_2+n\overline{x_2}).
\end{equation}
Splitting the right-hand side up the same way, we write
\begin{equation}
S_{\overline{\psi}^2}(m,n,p^r) = \sum_{\substack{x\pmod {p^r} \\ p\nmid x}}\psi^2(x)e_{p^r}(mx+n\overline{x}) = \sum_{x_1\pmod{p^{r-k}}}\sum_{\substack{x_2\pmod{p^k} \\ p\nmid x_2}}\psi^2(x_2)e_{p^r}(mp^kx_1+mx_2-np^kx_1\overline{x_2}^2 + n\overline{x_2}),
\end{equation}
which rearranges to
\begin{equation}
S_{\overline{\psi}^2}(m, n, p^r) = \sum_{\substack{x_2\pmod{p^k} \\ p\nmid x_2}}\psi^2(x_2)e_{p^r}(mx_2+n\overline{x_2})\sum_{x_1\pmod{p^{r-k}}}e_{p^{r-k}}\left(x_1(m-n\overline{x_2}^2)\right).
\end{equation}
Since the inner sum still vanishes unless $mx^2\equiv n\pmod{p^{r-k}}$, this sum collapses to
\begin{equation}
S_{\overline{\psi}^2}(m,n;c) = p^{r-k}\sum_{\substack{x_2\pmod{p^k} \\ mx_2^2\equiv n\pmod{p^{r-k}}}}\psi^2(x_2)e_{p^r}(mx_2+n\overline{x_2}),
\end{equation}
which is the same as the left-hand side, as desired. \\

If instead $r\ge 2k$, this result is a direct consequence of \eqref{cosetsum2k}.
\end{proof}

Substituting \eqref{ktwisted} into \eqref{sigmasplit} gives
\begin{corollary}
If $b,m,n$ are positive integers not divisible by $p$, and $c = bp^r$ where $r\ge j+k$, then
\begin{equation}
\sigma_{\psi,\epsilon}(m,n;c) = \psi(m)\overline{\psi}(n)\varphi(p^j)S_{\overline{\psi}^2}(m,n;c).
\end{equation}\label{corgjk}
\end{corollary}
Note that this agrees with \eqref{cosetsum2k} when $r\ge 2k$ by applying \eqref{kloosterman2k} with $\chi=\psi$. \\

\begin{remark}
Corollaries \ref{corljk} and \ref{corgjk} together can be thought of as a special case of Lemma 4.19 from \cite{YH}. In Hu's notation, $\mu$ is our $c^{-2}$, $k$ is our $r$, $i_0$ is our $k$, $c_0$ is our $p^k$ (per Hu's Definition 4.17, since we are doing the principal series case), $l$ is our $j$, and $l_0$ equals 0 in our work. The symbols $\theta$ and $\theta'$ are essentially Dirichlet characters in the principal series case, and the sum on the left runs over our coset $\psi\widehat{G}_{p^j}$. The group index $[\theta[l]: \theta[l_0]]$ equals $\varphi(p^j)$ in our case, and it can also be noted that there is no dependence on $\epsilon$ since the opposite-sign sum vanishes due to \ref{lemmakminus}.
\end{remark}

Substituting Lemma \ref{lemmakminus}, Corollary \ref{corljk}, and Corollary \ref{corgjk} in \eqref{sigmamn} gives the expansion
\begin{multline}
\sigma_{\psi,\epsilon}(m, n) = \frac{\pi(4\pi)^{\kappa-1}\left[\Gamma_0(1):\Gamma_0(q)\right]\psi(m)\overline{\psi}(n)\varphi(p^j)}{3\Gamma(\kappa-1)} \times \\
\left[\delta(m, n) + 2\pi i^{-\kappa}\epsilon\psi(-1)\sum_{\substack{b\ge 1 \\ p\nmid b}}\frac{S(\overline{q}m,\overline{q}n;b)\mathcal{K}_{\psi}^-(\overline{b}m,\overline{b}n;q)}{bq}J_{\kappa-1}\left(\frac{4\pi\sqrt{mn}}{bq}\right) \right. \\ + \left.2\pi i^{-\kappa}\sum_{\substack{c> 0 \\ c\equiv 0\pmod{p^{j+k}}}}\frac{S_{\overline{\psi}^2}(m,n;c)}cJ_{\kappa-1}\left(\frac{4\pi\sqrt{mn}}c\right)\right].\label{sigmamnsum}
\end{multline}

Finally, in the only case where the coset sum is not a Kloosterman sum, we can write it as a Fourier series in multiplicative characters:
\begin{lemma}
If $b, m$ are relatively prime positive integers not divisible by $p$, and $c = bq$, then
\begin{equation}
\sigma_{\psi,\epsilon}(m,1;c) = \frac{\epsilon}{\varphi(c)}\sum_{\substack{\eta\pmod c \\ \nu_p(\cond(\eta^2)) \le k-j}}\overline{\eta}(m)\widehat{\sigma}_{\psi}(\eta),\label{csfourier}
\end{equation}
where
\begin{equation}
\widehat{\sigma}_{\psi}(\eta) = \frac{\eta_q(-1)\eta_q^2(b)\eta_b^2(q)\tau^2(\eta_b)\tau(\psi\eta_q)\varphi(p^{j+k})}{\tau(\psi\overline{\eta_q})},\label{csfourierdef}
\end{equation}
and $\eta_b,\eta_q$ are characters modulo $b,q$ respectively whose product is $\eta$.
\end{lemma}

\begin{proof}
Using Fourier inversion, write
\begin{equation}
\widehat{\sigma}_{\psi}(\eta) = \epsilon\sum_{m\pmod c}\eta(m)\sigma_{\psi,\epsilon}(m,1;c),
\end{equation}
for any character $\eta\pmod c$. Opening up the sum in \eqref{cosetsum}, and using the fact that the same-sign sum vanishes due to Lemma \ref{lemmakplus}, gives
\begin{equation}
\widehat{\sigma}_{\psi}(\eta) = \sum_{m\pmod c}\sum_{\chi\in\psi\widehat{G}_{p^j}}\chi(-m)\eta(m)S_{\overline{\chi}^2}(m,1;c).
\end{equation}
Opening up the Kloosterman sum using \eqref{klfourier} gives
\begin{equation}
\widehat{\sigma}_{\psi}(\eta) = \frac1{\varphi(c)}\sum_{m\pmod c}\sum_{\chi\in\psi\widehat{G}_{p^j}}\sum_{\vartheta\pmod c}\chi(-m)\eta(m)\overline{\vartheta}(m)\tau(\vartheta)\tau(\overline{\chi}^2\vartheta).
\end{equation}
Now we bring the sum over $m$ to the inside and write
\begin{equation}
\widehat{\sigma}_{\psi}(\eta) = \frac1{\varphi(c)}\sum_{\chi\in\psi\widehat{G}_{p^j}}\sum_{\vartheta\pmod c}\chi(-1)\tau(\vartheta)\tau(\overline{\chi}^2\vartheta)\sum_{m\pmod c}\chi(m)\eta(m)\overline{\vartheta}(m).
\end{equation}
This sum equals $\varphi(c)$ if $\vartheta = \chi\eta$ and $0$ otherwise. Thus this collapses to give
\begin{equation}
\widehat{\sigma}_{\psi}(\eta) = \sum_{\chi\in\psi\widehat{G}_{p^j}}\chi(-1)\tau(\chi\eta)\tau(\overline{\chi}\eta).
\end{equation}
Due to equation (3.16) of \cite{IK}, we can factor the Gauss sums as
\begin{align}
\tau(\chi\eta) &= \chi(b)\eta_q(b)\eta_b(q)\tau(\chi\eta_q)\tau(\eta_b), \\
\tau(\overline{\chi}\eta) &= \overline{\chi}(b)\eta_q(b)\eta_b(q)\tau(\overline{\chi}\eta_q)\tau(\eta_b),
\end{align}
and so
\begin{equation}
\widehat{\sigma}_{\psi}(\eta) = \eta_q^2(b)\eta_b^2(q)\tau^2(\eta_b)\sum_{\chi\in\psi\widehat{G}_{p^j}}\chi(-1)\tau(\chi\eta_q)\tau(\overline{\chi}\eta_q).
\end{equation}
Now opening up the Gauss sums gives 
\begin{equation}
\widehat{\sigma}_{\psi}(\eta) = \eta_q^2(b)\eta_b^2(q)\tau^2(\eta_b)\sum_{\chi\in\psi\widehat{G}_{p^j}}\sum_{x\pmod q}\sum_{y\pmod q}\chi(-x)\overline{\chi}(y)\eta_q(x)\eta_q(y)e_q(x+y).
\end{equation}
Bring the sum over $\chi$ to the inside to get
\begin{equation}
\widehat{\sigma}_{\psi}(\eta) = \eta_q^2(b)\eta_b^2(q)\tau^2(\eta_b)\sum_{x\pmod q}\sum_{y\pmod q}\eta_q(x)\eta_q(y)e_q(x+y)\sum_{\chi\in\psi\widehat{G}_{p^j}}\chi(-x)\overline{\chi}(y).
\end{equation}
If $y\equiv -x\pmod{p^j}$, the inner sum equals $\psi(-x)\overline{\psi}(y)\varphi(p^j)$. Otherwise it vanishes. So we write
\begin{equation}
\widehat{\sigma}_{\psi}(\eta) = \eta_q^2(b)\eta_b^2(q)\tau^2(\eta_b)\varphi(p^j)\sum_{\substack{x, y\pmod q \\ y\equiv -x\pmod{p^j}}}\psi(-x)\overline{\psi}(y)\eta_q(x)\eta_q(y)e_q(x+y).
\end{equation}
Write $y = p^jt - x$, where $t$ runs modulo $p^{k-j}$. Then
\begin{equation}
\widehat{\sigma}_{\psi}(\eta) = \eta_q^2(b)\eta_b^2(q)\tau^2(\eta_b)\varphi(p^j)\sum_{x\pmod q}\psi(-x)\eta_q(x)\sum_{t\pmod{p^{k-j}}}\overline{\psi}(p^jt-x)\eta_q(p^jt-x)e_{p^{k-j}}(t).
\end{equation}
If $\cond(\psi\overline{\eta_q}) \neq q$, the sum over $t$ vanishes due to orthogonality. So suppose in the following that $\psi\overline{\eta_q}$ is primitive. Supposing also that $p\nmid x$, we have
\begin{equation}
\overline{\psi}(p^jt-x)\eta_q(p^jt-x) = \frac1{\tau(\psi\overline{\eta_q})}\sum_{u\pmod q}\psi(u)\overline{\eta_q}(u)e_q\left(u(p^jt-x)\right).
\end{equation}
Summing over $t$ gives
\begin{equation}
\sum_{t\pmod{p^{k-j}}}\overline{\psi}(p^jt-x)\eta_q(p^jt-x)e_{p^{k-j}}(t) = \frac1{\tau(\psi\overline{\eta_q})}\sum_{u\pmod q}\psi(u)\overline{\eta_q}(u)e_q(-ux)\sum_{t\pmod{p^{k-j}}}e_{p^{k-j}}\left(t(u+1)\right). 
\end{equation}
If $u\equiv -1\pmod{p^{k-j}}$, the sum over $t$ equals $p^{k-j}$; otherwise it vanishes. So we substitute $u = p^{k-j}v - 1$, where $v$ runs modulo $p^j$, and collapse to
\begin{equation}
\sum_{t\pmod{p^{k-j}}}\overline{\psi}(p^jt-x)\eta_q(p^jt-x)e_{p^{k-j}}(t) = \frac{p^{k-j}}{\tau(\psi\overline{\eta_q})}\sum_{v\pmod{p^j}}\psi(p^{k-j}v-1)\overline{\eta_q}(p^{k-j}v-1)e_q\left(-x(p^{k-j}v-1)\right).
\end{equation}
Then substituting this back in gives
\begin{equation}
\widehat{\sigma}_{\psi}(\eta) = \frac{\eta_q^2(b)\eta_b^2(q)\tau^2(\eta_b)\varphi(q)}{\tau(\psi\overline{\eta_q})}\sum_{x\pmod q}\psi(-x)\eta_q(x)\sum_{v\pmod{p^j}}\psi(p^{k-j}v-1)\overline{\eta_q}(p^{k-j}v-1)e_q\left(-x(p^{k-j}v-1)\right).
\end{equation}
The sum over $x$ is a Gauss sum, so we write
\begin{equation}
\widehat{\sigma}_{\psi}(\eta) = \frac{\eta_q(-1)\eta_q^2(b)\eta_b^2(q)\tau^2(\eta_b)\tau(\psi\eta_q)\varphi(q)}{\tau(\psi\overline{\eta_q})}\sum_{v\pmod{p^j}}\psi(p^{k-j}v-1)\overline{\eta_q}(p^{k-j}v-1)\overline{\psi}(p^{k-j}v-1)\overline{\eta_q}(p^{k-j}v-1),
\end{equation}
which collapses to
\begin{equation}
\widehat{\sigma}_{\psi}(\eta) = \frac{\eta_q(-1)\eta_q^2(b)\eta_b^2(q)\tau^2(\eta_b)\tau(\psi\eta_q)\varphi(q)}{\tau(\psi\overline{\eta_q})}\sum_{v\pmod{p^j}}\overline{\eta_q}^2(p^{k-j}v-1).
\end{equation}
The remaining sum equals $p^j$ (since all terms equal 1) if $\cond(\eta_q^2) \le p^{k-j}$; and otherwise it equals 0. Thus the result follows.
\end{proof}

Substituting in the principal character $\eta=\chi_0$ gives a simpler result:
\begin{corollary}
We have
\begin{equation}
\widehat{\sigma}_{\psi}(1) = \mu^2(b)\varphi(p^{j+k}).\label{csprincipal}
\end{equation}
\end{corollary}

Bounding the Gauss sums also gives a quick bound:
\begin{corollary}
We have
\begin{equation}
|\widehat{\sigma}_{\psi}(\eta)| \le bp^{j+k}.\label{sigfourierbound}
\end{equation}
\end{corollary}

\end{subsection}

\begin{subsection}{The Non-Diagonal Terms}
Now, as before, we want to study the non-diagonal terms in the explicit formula. Just like in our analysis of the thinnest family, we can take the terms where $e=2$ and evaluate them exactly using our previous work. We separate out the diagonal as in \eqref{diagsep} and apply \eqref{squad} to get
\begin{equation}
\mathcal{D}_1(\mathcal{F},\phi,R,w) = \widehat{\phi}(0) + \frac12\phi(0) + \frac2{\log R}\widehat{\phi}(0) + \frac2{\log R}\mathcal{I}(\kappa,\phi,R) - \frac2{\log R}\mathcal{J}(q,\phi,R) + \mathcal{E}(\mathcal{F},\phi, R),\label{diagsepcoset}
\end{equation}
where the non-diagonal terms are given by
\begin{equation}
\mathcal{E}(\mathcal{F},\phi,R) = -\frac2{\sigma_{\psi,\epsilon}(1,1)\log R}\sum_{\ell\neq p}\sum_{e\ge 1}\frac{\left[\sigma_{\psi,\epsilon}\left(\ell^e,1\right) - \delta_{e\ge 3}\sigma_{\psi,\epsilon}\left(\ell^{e-2},1\right)\right]\log\ell}{\ell^{e/2}}\widehat{\phi}\left(\frac{e\log \ell}{\log R}\right).
\end{equation}
Define
\begin{multline}
E_{\psi,\epsilon}(m) = 2\pi i^{-\kappa}\left[\epsilon\psi(-m)\sum_{\substack{b\ge 1 \\ p\nmid b}}\frac{S(\overline{q}m,\overline{q};b)\mathcal{K}_{\psi}^-(\overline{b}m,\overline{b};q)}{bq}J_{\kappa-1}\left(\frac{4\pi\sqrt{m}}{bq}\right)\right. \\ \left.  + \psi(m)\sum_{\substack{c> 0 \\ c\equiv 0\pmod{p^{j+k}}}}\frac{S_{\overline{\psi}^2}(m,1;c)}cJ_{\kappa-1}\left(\frac{4\pi\sqrt{m}}c\right)\right]\label{epsidef}
\end{multline}
to be the non-diagonal part of \eqref{sigmamnsum}, so that we may write
\begin{equation}
\frac{\sigma_{\psi,\epsilon}(m, 1)}{\sigma_{\psi,\epsilon}(1, 1)} = \frac{E_{\psi,\epsilon}(m)}{1 + E_{\psi,\epsilon}(1)},
\end{equation}
when $m \neq 1$. We can bound this sum as follows:

\begin{lemma}
For positive integers $m$, we have
\begin{equation}
E_{\psi,\epsilon}(m) \ll \min\left(m^{\frac14+\varepsilon}q^{-\frac12-\frac{j}k+\varepsilon}, m^{\frac12+\varepsilon}q^{-1-\frac{j}k+\varepsilon}\right),\label{epsimbound}
\end{equation}
where the implied constant depends only on $\kappa$ and $\varepsilon$.
\end{lemma}

\begin{proof}
To bound the sum over $b$ in \eqref{epsidef}, we write
\begin{equation}
\sum_{\substack{b\ge 1\\ p\nmid b}}\frac{S(\overline{q}m,\overline{q};b)\mathcal{K}_{\psi}^-(\overline{b}m,\overline{b};q)}{bq}J_{\kappa-1}\left(\frac{4\pi\sqrt{m}}{bq}\right) \ll \sum_{\substack{b\ge 1 \\ p\nmid b}}b^{-\frac12+\varepsilon}p^{-j}\min\left(m^{\frac12}b^{-1}q^{-1}, m^{-\frac14}b^{\frac12}q^{\frac12}\right),
\end{equation}
using \eqref{weil}, \eqref{jestx}, \eqref{jestsqrt}, and the bound $\mathcal{K}_{\psi}^{\pm}(m,n;p^r) \ll p^{r-j}$ which is simply due to the triangle inequality and Hensel's lemma. To bound the sum over $c$ in \eqref{epsidef}, meanwhile, we write
\begin{equation}
\sum_{\substack{c> 0 \\ c\equiv 0\pmod{p^{j+k}}}}\frac{S_{\overline{\psi}^2}(m,1;c)}cJ_{\kappa-1}\left(\frac{4\pi\sqrt{m}}c\right) \ll \sum_{\substack{c>0 \\ c\equiv 0\pmod{p^{j+k}}}}c^{-\frac12+\varepsilon}q^{\frac12}\min\left(m^{\frac12}c^{-1}, m^{-\frac14}c^{\frac12}\right),
\end{equation}
using \eqref{weilkl}, \eqref{jestx}, and \eqref{jestsqrt}.
The bound $m^{\frac14+\varepsilon}q^{-\frac12-\frac{j}k+\varepsilon}$ results from splitting up each sum based on which term of the min is smaller; and the bound $m^{\frac12+\varepsilon}q^{-1-\frac{j}k+\varepsilon}$ results from just taking the first term of the min (which is better if $\sqrt{m} < q$ since one of the two sums is empty in that case).
\end{proof}

In particular, taking $m=1$ gives the following useful result:
\begin{corollary}
We have
\begin{equation}
E_{\psi,\epsilon}(1) \ll q^{-1-\frac{j}k+\varepsilon},\label{epsi1bound}
\end{equation}
where the implied constant depends only on $\kappa$ and $\varepsilon$.
\end{corollary}

Thus
\begin{equation}
\mathcal{E}(\mathcal{F},\phi,R) = -\frac2{(1+E_{\psi,\epsilon}(1))\log R}\sum_{\ell\neq p}\sum_{e\ge 1}\frac{\left[E_{\psi,\epsilon}(\ell^e) - \delta_{e\ge 3}E_{\psi,\epsilon}(\ell^{e-2})\right]\log\ell}{\ell^{e/2}}\widehat{\phi}\left(\frac{e\log \ell}{\log R}\right).
\end{equation}
Now for positive integers $e$ and $e'$, define
\begin{equation}
T_{\mathcal{F}}(e, e') = \sum_{\ell\neq p}\frac{E_{\psi,\epsilon}(\ell^{e'})\log\ell}{\ell^{e/2}}\widehat{\phi}\left(\frac{e\log\ell}{\log R}\right),\label{tfdef}
\end{equation}
by analogy with \eqref{sfdef}, so that
\begin{equation}
\mathcal{E}(\mathcal{F},\phi,R) = -\frac2{(1 + E_{\psi,\epsilon}(1))\log R}\sum_{1\le e\le \theta\log_2(R)}\left[T_{\mathcal{F}}(e, e) - \delta_{e\ge 3}T_{\mathcal{F}}(e,e-2)\right].
\end{equation}

Now we want to develop a bound on $T_{\mathcal{F}}(e, e')$.
\begin{lemma}
We have
\begin{equation}
T_{\mathcal{F}}(e, e') \ll ||\widehat{\phi}||_{\infty}q^{-1-\frac{j}k+\varepsilon}\left(q^{-\frac{e}{e'}+1+\frac2{e'}} + q^{-\theta+\frac12+\frac{\theta e'}{2e}+\frac{2\theta}e}\right),\label{tfeebound}
\end{equation}
where the implied constant depends only on $\kappa$ and $\varepsilon$.
\end{lemma}

\begin{proof}
Substituting \eqref{epsimbound} into \eqref{tfdef} gives
\begin{equation}
T_{\mathcal{F}}(e, e') \ll ||\widehat{\phi}||_{\infty}\sum_{\substack{\ell\neq p \\ \ell \le R^{\theta/e}}}\min\left(\ell^{-\frac{e}2+\frac{e'}4+\varepsilon}q^{-\frac12-\frac{j}k+\varepsilon}, \ell^{-\frac{e}2+\frac{e'}2+\varepsilon}q^{-1-\frac{j}k+\varepsilon}\right).
\end{equation}
This splits up to give
\begin{equation}
T_{\mathcal{F}}(e, e') \ll ||\widehat{\phi}||_{\infty}q^{-1-\frac{j}k+\varepsilon}\Biggr[\sum_{\substack{\ell\neq p \\ \ell \le \min\left(q^{2/e'}, R^{\theta/e}\right)}}\ell^{-\frac{e}2+\frac{e'}2+\varepsilon} + q^{\frac12}\sum_{\substack{\ell\neq p \\ q^{2/e'} < \ell \le R^{\theta/e}}}\ell^{-\frac{e}2+\frac{e'}4+\varepsilon}\Biggr].
\end{equation}
We bound each sum to get
\begin{equation}
T_{\mathcal{F}}(e, e') \ll ||\widehat{\phi}||_{\infty}q^{-1-\frac{j}k+\varepsilon}\left[1 + \left(q^{\frac2{e'}}\right)^{-\frac{e}2+\frac{e'}2+1+\varepsilon} + q^{\frac12}\left(q^{\frac{2\theta}e}\right)^{-\frac{e}2+\frac{e'}4+1+\varepsilon}\right],
\end{equation}
which gives the result upon expanding out the exponents.
\end{proof}
Specializing $e'=e$ or $e'=e-2$ gives
\begin{corollary}
For $e\ge 1$, we have
\begin{equation}
T_{\mathcal{F}}(e, e) \ll ||\widehat{\phi}||_{\infty}q^{-1-\frac{j}k+\varepsilon}\left(q^{\frac2e} + q^{-\frac{\theta}2+\frac{2\theta}e+\frac12}\right),\label{tfe1bound}
\end{equation}
where the implied constant depends only on $\kappa$ and $\varepsilon$.
\end{corollary}
\begin{corollary}
For $e\ge 2$, we have
\begin{equation}
T_{\mathcal{F}}(e, e) \ll ||\widehat{\phi}||_{\infty}q^{\frac{\max(\theta,1)-1}2-\frac{j}k+\varepsilon},\label{tfe2bound}
\end{equation}
where the implied constant depends only on $\kappa$ and $\varepsilon$.
\end{corollary}
\begin{corollary}
For $e\ge 3$, we have
\begin{equation}
T_{\mathcal{F}}(e, e-2) \ll ||\widehat{\phi}||_{\infty}q^{-\frac12-\frac{j}k+\varepsilon},\label{tfe3bound}
\end{equation}
where the implied constant depends only on $\kappa$ and $\varepsilon$.
\end{corollary}
Applying \eqref{epsi1bound}, \eqref{tfe2bound}, and \eqref{tfe3bound} gives a result that lets us place $T_{\mathcal{F}}(e,e')$ into the error term in \eqref{result2} in all cases except when $e=e'=1$:
\begin{corollary}
We have
\begin{equation}
\mathcal{E}(\mathcal{F},\phi, R) = -\frac{2T_{\mathcal{F}}(1, 1)}{\log R} + O\left(||\widehat{\phi}||_{\infty}q^{\frac{\max(\theta,1)-1}2-\frac{j}k+\varepsilon}\right),
\end{equation}
where the implied constant depends only on $\kappa$ and $\varepsilon$.
\end{corollary}

\end{subsection}

\begin{subsection}{The Error Term}
As before, we want to use GRH to bound our remaining error terms. To do this, we group together all the terms where $e = e'$, giving
\begin{equation}
\mathcal{E}(\mathcal{F},\phi, R) = -\frac2{\log R}\sum_{m\ge 1}\frac{\Lambda(m)E_{\psi,\epsilon}(m)}{\sqrt{m}}\widehat{\phi}\left(\frac{\log m}{\log R}\right) + O\left(||\widehat{\phi}||_{\infty}q^{\frac{\max(\theta,1)-1}2-\frac{j}k+\varepsilon}\right).
\end{equation}
We claim the result is
\begin{lemma}
\begin{equation}
\mathcal{E}(\mathcal{F},\phi, R) = -\frac{4\pi i^{-\kappa}\epsilon}{\log R}\sum_{\substack{b\ge 1 \\ p\nmid b \\ \text{squarefree}}}\frac{Q_{\chi_0}(bq)}{b\varphi(bq)} + O\left(\left(||\widehat{\phi}||_{\infty} + ||\widehat{\phi}\thinspace''||_{\infty}\right)q^{\max(\theta,1)-1-\frac{j}k+\varepsilon}\right).\label{nondiag}
\end{equation}
\end{lemma}
To prove this, we first open up the sum in \eqref{epsidef} to get
\begin{equation}
\mathcal{E}(\mathcal{F},\phi, R) = -\frac{4\pi i^{-\kappa}}{\log R}\left[\mathcal{E}_1(\mathcal{F},\phi,R) + \mathcal{E}_2(\mathcal{F},\phi,R)\right] + O\left(||\widehat{\phi}||_{\infty}q^{\frac{\max(\theta,1)-1}2-\frac{j}k+\varepsilon}\right),\label{e12split}
\end{equation}
where
\begin{equation}
\mathcal{E}_1(\mathcal{F},\phi,R) = \frac1{\varphi(p^j)}\sum_{m\ge 1}\sum_{\substack{b\ge 1 \\ p\nmid b}}\frac{\Lambda(m)\sigma_{\psi,\epsilon}(m,1;bq)}{bq\sqrt{m}}J_{\kappa-1}\left(\frac{4\pi\sqrt{m}}{bq}\right)\widehat{\phi}\left(\frac{\log m}{\log R}\right),
\end{equation}
and
\begin{equation}
\mathcal{E}_2(\mathcal{F},\phi,R) = \sum_{m\ge 1}\sum_{\substack{c>0 \\ c\equiv 0\pmod{p^{j+k}}}}\frac{\Lambda(m)\psi(m)S_{\overline{\psi}^2}(m,1;c)}{c\sqrt{m}}J_{\kappa-1}\left(\frac{4\pi\sqrt{m}}c\right)\widehat{\phi}\left(\frac{\log m}{\log R}\right).
\end{equation}
Then we bound each piece separately.

\begin{lemma}
Assume GRH, and take $B = q^{100}$. Then
\begin{equation}
\mathcal{E}_1(\mathcal{F},\phi,R) = \epsilon\sum_{\substack{1\le b\le B \\ p\nmid b \\ \text{squarefree}}}\frac{Q_{\chi_0}(bq)}{b\varphi(bq)} + O\left(\left(||\widehat{\phi}||_{\infty} + ||\widehat{\phi}\thinspace''||_{\infty}\right)q^{\max(\theta,1)-1-\frac{j}k+\varepsilon}\right),
\end{equation}
where the implied constant depends only on $\kappa$ and $\varepsilon$.
\end{lemma}

\begin{proof}
As before, we bound the tail of the series by
\begin{equation}
\frac1{\varphi(p^j)}\sum_{m\ge 1}\sum_{\substack{b>B \\ p\nmid b}}\frac{\Lambda(m)\sigma_{\psi,\epsilon}(m,1;bq)}{bq\sqrt{m}}J_{\kappa-1}\left(\frac{4\pi\sqrt{m}}{bq}\right)\widehat{\phi}\left(\frac{\log m}{\log R}\right) \ll ||\widehat{\phi}||_{\infty}B^{-\frac12+\varepsilon}q^{2\theta-1+\varepsilon},\label{bqtail}
\end{equation}
and we bound the terms where $(bq, m) \neq 1$ by
\begin{equation}
\frac1{\varphi(p^j)}\sum_{\substack{b\ge 1 \\ p\nmid b}}\sum_{\ell\mid b}\sum_{1\le e\le\theta\log_2(R)}\frac{\sigma_{\psi,\epsilon}(\ell^e,1;bq)\log\ell}{bq\ell^{e/2}}J_{\kappa-1}\left(\frac{4\pi\ell^{e/2}}{bq}\right)\widehat{\phi}\left(\frac{e\log\ell}{\log R}\right) \ll ||\widehat{\phi}||_{\infty}q^{-1+\varepsilon}.
\end{equation}
Using \eqref{csfourier}, we thus write
\begin{equation}
\mathcal{E}_1(\mathcal{F},\phi,R) = \frac{\epsilon}{\varphi(p^j)}\sum_{\substack{1\le b\le B \\ p\nmid b}}\sum_{\substack{\eta\pmod{bq} \\ \nu_p(\cond(\eta^2)) \le k-j}}\frac{\widehat{\sigma}_{\psi}(\overline{\eta})Q_{\eta}(bq)}{bq\varphi(bq)} + O\left(||\widehat{\phi}||_{\infty}B^{-\frac12+\varepsilon}q^{2\theta-1+\varepsilon} + ||\widehat{\phi}||_{\infty}q^{-1+\varepsilon}\right),
\end{equation}
where $Q_{\overline{\chi}\psi}(c)$ is as defined in \eqref{qcdef}. Due to \eqref{csprincipal}, the $\eta=\chi_0$ term gives
\begin{equation}
\epsilon\sum_{\substack{1\le b\le B \\ p\nmid b \\ \text{squarefree}}}\frac{Q_{\chi_0}(bq)}{b\varphi(bq)}.
\end{equation}
We can bound the nonprincipal terms using \eqref{qcbound} and \eqref{sigfourierbound}, writing
\begin{equation}
\frac{\epsilon}{\varphi(p^j)}\sum_{\substack{1\le b\le B \\ p\nmid b}}\sum_{\substack{\eta\pmod{bq} \\ \nu_p(\cond(\eta^2)) \le k-j \\ \eta\neq 1}}\frac{\widehat{\sigma}_{\psi}(\overline{\eta})Q_{\eta}(bq)}{bq\varphi(bq)} \ll \sum_{\substack{1\le b\le B \\ p\nmid b}}\sum_{\substack{\eta\pmod{bq} \\ \nu_p(\cond(\eta^2)) \le k-j \\ \eta\neq 1}}||\widehat{\phi}\thinspace''||_{\infty}b^{-2+\varepsilon}q^{\theta-2+\varepsilon} \ll ||\widehat{\phi}\thinspace''||_{\infty}B^{\varepsilon}q^{\theta-1-\frac{j}k+\varepsilon}.
\end{equation}
Thus
\begin{equation}
\mathcal{E}_1(\mathcal{F},\phi,R) = \epsilon\sum_{\substack{1\le b\le B \\ p\nmid b \\ \text{squarefree}}}\frac{Q_{\chi_0}(bq)}{b\varphi(bq)} + O\left(||\widehat{\phi}\thinspace''||_{\infty}B^{\varepsilon}q^{\theta-1-\frac{j}k+\varepsilon} + ||\widehat{\phi}||_{\infty}B^{-\frac12+\varepsilon}q^{2\theta-1+\varepsilon} + ||\widehat{\phi}||_{\infty}q^{-1+\varepsilon}\right).
\end{equation}
Taking $B$ to be a large fixed power of $q$, such as $q^{100}$, then gives the result.
\end{proof}

The bound for $\mathcal{E}_2(\mathcal{F},\phi,R)$ is completely analogous.

\begin{lemma}
Assume GRH. Then
\begin{equation}
\mathcal{E}_2(\mathcal{F},\phi,R) \ll \left(||\widehat{\phi}||_{\infty}+||\widehat{\phi}\thinspace''||_{\infty}\right)q^{\theta-1-\frac{j}k+\varepsilon},
\end{equation}
where the implied constant depends only on $\kappa$ and $\varepsilon$.
\end{lemma}

\begin{proof}
As before, we bound the tail of the series by
\begin{equation}
\sum_{m\ge 1}\sum_{\substack{c>C \\ c\equiv 0\pmod{p^{j+k}}}}\frac{\Lambda(m)\psi(m)S_{\overline{\psi}^2}(m,1;c)}{c\sqrt{m}}J_{\kappa-1}\left(\frac{4\pi\sqrt{m}}c\right)\widehat{\phi}\left(\frac{\log m}{\log R}\right) \ll ||\widehat{\phi}||_{\infty}C^{-\frac12+\varepsilon}q^{2\theta-\frac12-\frac{j}k+\varepsilon},
\end{equation}
and we bound the terms where $(c, m) \neq 1$ by
\begin{equation}
\sum_{\substack{c>0 \\ c\equiv 0\pmod{p^{j+k}}}}\sum_{\ell\mid c}\sum_{1\le e\le\theta\log_2(R)}\frac{\psi(\ell^e)S_{\overline{\psi}^2}(\ell^e,1;c)\log\ell}{c\ell^{e/2}}J_{\kappa-1}\left(\frac{4\pi\ell^{e/2}}c\right)\widehat{\phi}\left(\frac{e\log\ell}{\log R}\right) \ll ||\widehat{\phi}||_{\infty}q^{-1-\frac{j}k+\varepsilon}.
\end{equation}
Using \eqref{klfourier}, we thus write
\begin{equation}
\mathcal{E}_2(\mathcal{F},\phi,R) = \sum_{\substack{0<c\le C \\ c\equiv 0\pmod{p^{j+k}}}}\sum_{\chi\pmod c}\frac{\tau(\chi)\tau(\chi\overline{\psi}^2)Q_{\overline{\chi}\psi}(c)}{c\varphi(c)} + O\left(||\widehat{\phi}||_{\infty}C^{-\frac12+\varepsilon}q^{2\theta-\frac12-\frac{j}k+\varepsilon} + ||\widehat{\phi}||_{\infty}q^{-1-\frac{j}k+\varepsilon}\right),
\end{equation}
where $Q_{\overline{\chi}\psi}(c)$ is as defined in \eqref{qcdef}. Bounding the Gauss sums by $\sqrt{c}$ and using \eqref{qcbound} gives
\begin{equation}
\sum_{\substack{0<c\le C \\ c\equiv 0\pmod{p^{j+k}}}}\sum_{\chi\pmod c}\frac{\tau(\chi)\tau(\chi\overline{\psi}^2)Q_{\overline{\chi}\psi}(c)}{c\varphi(c)} \ll ||\widehat{\phi}\thinspace''||_{\infty}R^{\frac{\theta}2+\varepsilon}\sum_{\substack{0<c\le C \\ c\equiv 0\pmod{p^{j+k}}}}\sum_{\chi\pmod c}c^{-2+\varepsilon} \ll ||\widehat{\phi}\thinspace''||_{\infty}C^{\varepsilon}q^{\theta-1-\frac{j}k+\varepsilon},
\end{equation}
noting that the term where $\chi = \psi$ vanishes because $c$ is divisible by strictly more copies of $p$ than $\cond(\chi) = q$ is, so $\tau(\chi) = 0$. Thus
\begin{equation}
\mathcal{E}_2(\mathcal{F},\phi,R) \ll ||\widehat{\phi}\thinspace''||_{\infty}C^{\varepsilon}q^{\theta-1-\frac{j}k+\varepsilon} + ||\widehat{\phi}||_{\infty}C^{-\frac12+\varepsilon}q^{2\theta-\frac12-\frac{j}k+\varepsilon} + ||\widehat{\phi}||_{\infty}q^{-1-\frac{j}k+\varepsilon}.
\end{equation}
Taking $C$ to be a large fixed power of $q$, such as $q^{100}$, then gives the result.
\end{proof}

Substituting these both in to \eqref{e12split}, and adding the terms where $b > B$ back in since they fit into the error term, gives \eqref{nondiag}.

\end{subsection}

\begin{subsection}{The Special Main Term}
The non-diagonal terms include one sum which contains a significant main term that will help us distinguish between the symmetry groups SO(even) and SO(odd). In this section, we will extract that main term. First, substituting \eqref{nondiag} back into \eqref{diagsepcoset} gives
\begin{multline}
\mathcal{D}_1(\mathcal{F},\phi,R,w) = \widehat{\phi}(0) + \frac12\phi(0) + \frac2{\log R}\widehat{\phi}(0) + \frac2{\log R}\mathcal{I}(\kappa,\phi,R) - \frac2{\log R}\mathcal{J}(q,\phi,R) \\
-\epsilon\mathcal{M}(\mathcal{F},\phi,R) + O\left(\left(||\widehat{\phi}||_{\infty} + ||\widehat{\phi}\thinspace''||_{\infty}\right)q^{\max(\theta,1)-1-\frac{j}k+\varepsilon}\right),\label{mdiagsepcoset}
\end{multline}
where
\begin{equation}
\mathcal{M}(\mathcal{F},\phi, R) = \frac{4\pi i^{-\kappa}}{\log R}\sum_{\substack{1\le b\le B \\ p\nmid b \\ \text{squarefree}}}\frac{Q_{\chi_0}(bq)}{b\varphi(bq)}.\label{specmain}
\end{equation}
Here \eqref{specmain} contains a lower-order main term.

\begin{lemma}
We have
\begin{equation}
\mathcal{M}(\mathcal{F},\phi, R) = \frac{i^{\kappa}}2\left[\phi(0) - \int_{-1}^1\widehat{\phi}(t)\;dt\right] + \frac{2i^{\kappa}}{\log R}\mathcal{L}(\kappa,p,\phi, R) + O\left(\left(||\widehat{\phi}||_{\infty} + ||\widehat{\phi}\thinspace''||_{\infty}\right)q^{\theta-2+\varepsilon}\right),
\end{equation}
where
\begin{equation}
\mathcal{L}(\kappa,p, \phi, R) = \int_{-\infty}^{\infty}e(-y)\mathcal{J}_{\kappa,p}\left(\frac{4\pi iy}{\log R}\right)\phi(y)\;dy,\label{ldef} 
\end{equation}
and $\mathcal{J}_{\kappa,p}(s)$ is an analytic function depending on $\kappa$ and $p$ and defined in \eqref{jkpdef}.
\end{lemma}

\begin{proof}
First we apply \eqref{qcpbound} to write
\begin{equation}
\mathcal{M}(\mathcal{F},\phi, R) = \frac{4\pi i^{-\kappa}}{\log R}\sum_{\substack{b\ge 1 \\ p\nmid b \\ \text{squarefree}}}\frac{Q_M(bq)}{b\varphi(bq)} + O\left(\left(||\widehat{\phi}||_{\infty} + ||\widehat{\phi}\thinspace''||_{\infty}\right)q^{\theta-2+\varepsilon}\right),\label{specmainqm}
\end{equation}
where
\begin{equation}
Q_M(c) = \int_0^{\infty}\frac1{\sqrt x}J_{\kappa-1}\left(\frac{4\pi\sqrt{x}}c\right)\widehat{\phi}\left(\frac{\log x}{\log R}\right)\;dx.\label{qmdef}
\end{equation}
We substitute $y = \frac{4\pi\sqrt x}c$ to get
\begin{equation}
Q_M(c) = \frac{c}{2\pi}\int_0^{\infty}J_{\kappa-1}(y)\widehat{\phi}\left(\frac{2\log(cy/4\pi)}{\log R}\right)\;dy.
\end{equation}
Now we apply Mellin inversion. Let
\begin{equation}
\widetilde{J}_{\kappa-1}(s) = \int_0^{\infty}x^sJ_{\kappa-1}(x)\;\frac{dx}x = 2^{s-1}\frac{\Gamma\left(\frac{\kappa+s-1}2\right)}{\Gamma\left(\frac{\kappa-s+1}2\right)},
\end{equation}
be the Mellin transform of the Bessel function, where the integral is found using (6.561.14) of \cite{GR}. The region of convergence is $-\kappa+1 < \Re(s) < \frac32$. In this region, we write
\begin{equation}
J_{\kappa-1}(y) = \frac1{2\pi i}\int_{(\sigma)}y^{-s}\widetilde{J}_{\kappa-1}(s)\;ds.
\end{equation}
Substituting this in gives
\begin{equation}
Q_M(c) = \frac{c}{4\pi^2i}\int_0^{\infty}\int_{(\sigma)}y^{-s}\widetilde{J}_{\kappa-1}(s)\widehat{\phi}\left(\frac{2\log(cy/4\pi)}{\log R}\right)\;ds\;dy.
\end{equation}
We rearrange the order of integration to get
\begin{equation}
Q_M(c) = \frac{c}{4\pi^2i}\int_{(\sigma)}\widetilde{J}_{\kappa-1}(s)\int_0^{\infty}y^{-s}\widehat{\phi}\left(\frac{2\log(cy/4\pi)}{\log R}\right)\;dy\;ds.
\end{equation}
Write $y = \frac{4\pi R^{t/2}}c$ to transform this into
\begin{equation}
Q_M(c) = \frac{\log R}{2\pi i}\int_{(\sigma)}\widetilde{J}_{\kappa-1}(s)\int_{-\infty}^{\infty}\left(\frac{4\pi}c\right)^{-s}R^{-(s-1)t/2}\widehat{\phi}(t)\;dt\;ds.
\end{equation}
Write $R^{-(s-1)t} = e\left(\frac{i(s-1)t\log R}{2\pi}\right)$. Then using Fourier inversion, we have
\begin{equation}
\int_{-\infty}^{\infty}R^{-(s-1)t/2}\widehat{\phi}(t)\;dt = \phi\left(\frac{i(s-1)\log R}{4\pi}\right).
\end{equation}
So we have, substituting this in together with the formula for the Mellin transform of the Bessel function,
\begin{equation}
Q_M(c) = \frac{\log R}{2\pi i}\int_{(\sigma)}\left(\frac{4\pi}c\right)^{-s}\widetilde{J}_{\kappa-1}(s)\phi\left(\frac{i(s-1)\log R}{4\pi}\right)\;ds.
\end{equation}
Replacing $s$ with $1-s$ gives the formula
\begin{equation}
Q_M(c) = \frac{c\log R}{8\pi^2i}\int_{(\sigma)}\left(\frac{4\pi}c\right)^s\widetilde{J}_{\kappa-1}(1-s)\phi\left(\frac{is\log R}{4\pi}\right)\;ds,
\end{equation}
for $-\frac12 < \sigma < \kappa$. Now, we substitute this formula back into \eqref{specmainqm} to get
\begin{multline}
\mathcal{M}(\mathcal{F},\phi, R) = \frac{i^{-\kappa-1}}{2\pi\left(1-\frac1p\right)}\sum_{\substack{b\ge1 \\ p\nmid b \\ \text{squarefree}}}\frac1{\varphi(b)}\int_{(\sigma)}\left(\frac{4\pi}{bq}\right)^s\widetilde{J}_{\kappa-1}(1-s)\phi\left(\frac{is\log R}{4\pi}\right)\;ds \\
+ O\left(\left(||\widehat{\phi}||_{\infty} + ||\widehat{\phi}\thinspace''||_{\infty}\right)q^{\theta-2+\varepsilon}\right).
\end{multline}
Now for $\Re(s) > 1$, define the special function
\begin{equation}
\xi(s) = \sum_{\substack{b\ge 1 \\ \text{squarefree}}}\frac1{b^{s-1}\varphi(b)} = \prod_{\ell}\left(1+\frac1{\ell^{s-1}(\ell-1)}\right) = \zeta(s)A(s),
\end{equation}
where
\begin{equation}
A(s) = \prod_{\ell}\left(1 + \frac{\ell^{s-1}-1}{\ell^{2s-1}(\ell-1)}\right),
\end{equation}
and the product is over all primes $\ell$. Then for $\sigma > 0$, we have
\begin{equation}
\mathcal{M}(\mathcal{F},\phi, R) = \frac{i^{-\kappa-1}p}{2\pi}\int_{(\sigma)}\left(\frac{4\pi p}{q}\right)^s\frac{\xi(1+s)\widetilde{J}_{\kappa-1}(1-s)}{p^s(p-1)+1}\phi\left(\frac{is\log R}{4\pi}\right)\;ds + O\left(\left(||\widehat{\phi}||_{\infty} + ||\widehat{\phi}\thinspace''||_{\infty}\right)q^{\theta-2+\varepsilon}\right).
\end{equation}
Now we want to consider what happens when we shift the contour to approach $\sigma = 0$ from the right. All the factors in the integrand are holomorphic in a neighborhood of the real line, except for $\xi(1+s)$. Now $\zeta(s)$ is holomorphic everywhere away from a simple pole at $s=1$. Meanwhile, the Euler product $A(s)$ for $\Re(s) > \frac12$, so $\xi(s)$ is holomorphic everywhere in the half-plane $\Re(s) > \frac12$ except for a simple pole at $s=1$. We also have $A(1) = 1$ and
\begin{equation}
\widetilde{J}_{\kappa-1}(1-s) = 2^{-s}\frac{\Gamma\left(\frac{\kappa-s}2\right)}{\Gamma\left(\frac{\kappa+s}2\right)},
\end{equation}
which is analytic in a neighborhood of the origin and satisfies $\widetilde{J}_{\kappa-1}(1-s) = 1$ when $s = 0$. Putting these all together, we define the function
\begin{equation}
\mathcal{J}_{\kappa,p}(s) = \frac{p^{1+s}\xi(1+s)}{p^s(p-1)+1}\frac{\Gamma\left(\frac{\kappa-s}2\right)}{\Gamma\left(\frac{\kappa+s}2\right)} - \frac1s,\label{jkpdef}
\end{equation}
which is analytic in the strip $-\frac12 < \Re(s) < \kappa$, and write
\begin{equation}
\mathcal{M}(\mathcal{F},\phi, R) = \frac{i^{-\kappa-1}}{2\pi}\int_{(\sigma)}\left(\frac{2\pi}{q}\right)^s\left[\frac1s + \mathcal{J}_{\kappa,p}(s)\right]\phi\left(\frac{is\log R}{4\pi}\right)\;ds + O\left(\left(||\widehat{\phi}||_{\infty} + ||\widehat{\phi}\thinspace''||_{\infty}\right)q^{\theta-2+\varepsilon}\right).
\end{equation}
Writing $s$ in place of $\frac{s\log R}{4\pi}$ and using the fact that $R = \log\left(\frac{q^2}{4\pi^2}\right)$ gives
\begin{multline}
\mathcal{M}(\mathcal{F},\phi, R) = \frac{i^{-\kappa-1}}{2\pi}\int_{(\sigma)}\frac{e^{-2\pi s}\phi(is)}s\;ds + \frac{2i^{-\kappa-1}}{\log R}\int_{(\sigma)}e^{-2\pi s}\mathcal{J}_{\kappa,p}\left(\frac{4\pi s}{\log R}\right)\phi(is)\;ds \\
+ O\left(\left(||\widehat{\phi}||_{\infty} + ||\widehat{\phi}\thinspace''||_{\infty}\right)q^{\theta-2+\varepsilon}\right).
\end{multline}
In the second integral above, we can send $\sigma\to 0$ and write $s = iy$ to get
\begin{equation}
\frac{2i^{-\kappa-1}}{\log R}\int_{(\sigma)}e^{-2\pi s}\mathcal{J}_{\kappa,p}\left(\frac{4\pi s}{\log R}\right)\phi(is)\;ds = \frac{2i^{\kappa}}{\log R}\int_{-\infty}^{\infty}e(-y)\mathcal{J}_{\kappa,p}\left(\frac{4\pi iy}{\log R}\right)\phi(y)\;dy = \frac{2i^{\kappa}}{\log R}\mathcal{L}(\kappa,p,\phi,R).
\end{equation}
It remains to evaluate the main term
\begin{equation}
\mathcal{N}(\kappa, \phi) = \frac{i^{-\kappa-1}}{2\pi}\int_{(\sigma)}\frac{e^{-2\pi s}\phi(is)}s\;ds.
\end{equation}
First, we write $s = \sigma + it$ to get
\begin{equation}
\mathcal{N}(\kappa, \phi) = \frac{i^{\kappa}e^{-2\pi\sigma}}{2\pi}\int_{-\infty}^{\infty}\frac{e^{-2\pi it}\phi(i\sigma-t)}{\sigma+it}\;dt.
\end{equation}
Now, we send $\sigma$ to 0 from the right and split the complex exponential up into a cosine and a sine, thus writing
\begin{equation}
\mathcal{N}(\kappa, \phi) = \frac{i^{\kappa}}{2\pi}\left[\mathcal{C}(\phi) - i\mathcal{S}(\phi)\right],
\end{equation}
where
\begin{equation}
\mathcal{C}(\phi) = \lim_{\sigma\to 0^+}\int_{-\infty}^{\infty}\frac{\cos(2\pi t)\phi(i\sigma-t)}{\sigma+it}\;dt \qquad\text{and}\qquad \mathcal{S}(\phi) = \lim_{\sigma\to 0^+}\int_{-\infty}^{\infty}\frac{\sin(2\pi t)\phi(i\sigma-t)}{\sigma+it}\;dt.
\end{equation}
Since $\sin(t)$ vanishes at $t=0$, we can set $\sigma$ directly to 0 in $\mathcal{S}(\phi)$ to get
\begin{equation}
\mathcal{S}(\phi) = -i\int_{-\infty}^{\infty}\frac{\sin(2\pi t)\phi(t)}t\;dt.
\end{equation}
Now the Fourier transform of $f(t) = \frac{\sin(2\pi t)}t$ is
\begin{equation}
\widehat{f}(\xi) = \begin{cases}
\pi & (|\xi| < 1) \\
0 & (|\xi| > 1)
\end{cases}.
\end{equation}
Thus by Plancherel's theorem,
\begin{equation}
\mathcal{S}(\phi) = -i\int_{-\infty}^{\infty}f(t)\phi(t)\;dt = -i\int_{-\infty}^{\infty}\widehat{f}(\xi)\widehat{\phi}(\xi)\;d\xi = -i\pi\int_{-1}^1\widehat{\phi}(t)\;dt.
\end{equation}
Meanwhile, we cannot just set $\sigma$ to 0 in $\mathcal{C}(\phi)$ because there is a pole at $s=0$. However, we can set $\sigma$ almost to zero with the exception of a semicircle of radius $\varepsilon$ around the origin. This gives
\begin{equation}
\mathcal{C}(\phi) = \int_{-\infty}^{-\varepsilon}\cos(2\pi t)\phi(t)\frac{dt}t + \int_{\varepsilon}^{\infty}\cos(2\pi t)\phi(t)\frac{dt}t + \int_{-\pi/2}^{\pi/2}e^{-2\pi \varepsilon\cos\theta}\cos(2\pi \varepsilon\sin\theta)\phi(i\varepsilon)\;d\theta.
\end{equation}
The first two integrals cancel out precisely. In the third, we can now send $\varepsilon$ to 0 to get
\begin{equation}
\mathcal{C}(\phi) = \int_{-\pi/2}^{\pi/2}\phi(0)\;d\theta = \pi\phi(0).
\end{equation}
Putting these together gives
\begin{equation}
\mathcal{N}(\kappa,\phi) = \frac{i^{\kappa}}2\left[\phi(0) - \int_{-1}^1\widehat{\phi}(t)\;dt\right],\label{inm1}
\end{equation}
and the result follows.
\end{proof}
\begin{remark}
The symmetry type and remaining lower-order main terms in Theorem \ref{theorem2} follow from substituting this back into \eqref{mdiagsepcoset}.
\end{remark}
\end{subsection}
\end{section}

\begin{section}{Analogy with the Weight Aspect} \label{analogysection}

\stepcounter{subsection}

It is useful to consider the analogy of the work above, which is done in level aspect, with the analogous problem of thin families in weight aspect. A useful reference is \cite{DFS}, which studied in the weight aspect a family analogous to the thinnest family in this work. The low-lying zeros problem in weight aspect for intermediate size families is to find the asymptotics of
\begin{equation}
\mathcal{D}_{\epsilon}(T, \Delta, h, \phi) = \frac{\sum_{\kappa\equiv 3+\epsilon\pmod 4}h\left(\frac{\kappa-T}{\Delta}\right)\mathcal{D}_1(\mathcal{H}_{\kappa}(1,1), \phi, T^2, w)}{\sum_{\kappa\equiv 3+\epsilon\pmod 4}h\left(\frac{\kappa-T}{\Delta}\right)},
\end{equation}
where $\epsilon = \pm1$, $T$ and $\Delta$ are positive real numbers with $T^{\varepsilon} \le \Delta \le T^{1-\varepsilon}$, $h$ is a smooth nonnegative weight function with compact support contained in $(0, \infty)$, $\phi$ is a real even Schwartz-class function as before, and $w(f) = \frac1{\left<f,f\right>_1}$ is the Petersson weight. Here we are thinking of $\Delta$ and $T$ as growing large (and thus the weight of the cuspforms being averaged as growing large) while the level stays at $q=1$. The first observation we can make is that $T$ and $\Delta$ are analogous to $p^k$ and $p^j$ respectively in the coset family, since $T$ controls the minimum value of the weight $\kappa$ (analogous to the level $q$ in our work) and $\Delta$ controls how ``spread out" the average is, since larger values cause there to be more weights being averaged. This reflects how larger values of $j$ (while still less than $k$) enlarge the coset family $\mathcal{F}_{\kappa,\epsilon}(q,p^j,\psi)$ and thus allow us to average over a larger set. \\

Just like the thinnest family can often be thought of as the limit of the coset family as $j\to 0$, the analogue of this family in level aspect is a sort of limit as $\Delta \to 1$. Namely, the average is taken over only a single weight $\kappa$ of size $T$. The first result in this direction is Theorem 1.2 of \cite{ILS}, which gives support up to $(-1, 1)$, matching the support in our Theorem \ref{theorem1}. Meanwhile, Theorem 1.3 of \cite{ILS} is the other extreme $\Delta \to T$, and it accordingly gives support up to $(-2, 2)$. However, there do not seem to be many results in the literature about the low-lying zeros of the intermediate families in weight aspect where $T^{\varepsilon} \le \Delta \le T^{1-\varepsilon}$. We expect that the support would go up to $1 + \frac{\log\Delta}{\log T}$ in this case, in analogy with our results in the level aspect. \\

Using our explicit formula \eqref{olds}, we can write analogously to \eqref{oldavg1}
\begin{equation}
\mathcal{D}_1(\mathcal{B}, \phi, T^2, w) = \frac2{\log R}\mathcal{I}(\kappa,\phi,T^2) - \sum_{\ell}\sum_{e\ge 1}\frac{\left[\Delta_{\mathcal{B}}(\ell^e) - \Delta_{\mathcal{B}}(\ell^{e-2})\right]\log \ell}{2\ell^{e/2}\Delta_{\mathcal{B}}(1)\log T}\widehat{\phi}\left(\frac{e\log\ell}{2\log T}\right),
\end{equation}
where $\mathcal{B} = \mathcal{H}_{\kappa}(1, 1)$. Now we follow the presentation of Section 5.5 of \cite{IW}. Using the Petersson trace formula to open up $\Delta_{\mathcal{B}}(\ell^e)$ and related sums, we can bring the sum over $\kappa$ to the inside to obtain a sum of the form
\begin{equation}
G(m, n) = 2\pi\sum_{c\ge 1}c^{-1}G\left(\frac{4\pi\sqrt{mn}}c\right)S(m,n;c),
\end{equation}
where
\begin{equation}
G(x) = \epsilon\sum_{\kappa\equiv 3+\epsilon\pmod 4}h\left(\frac{\kappa-T}{\Delta}\right)J_{\kappa-1}(x),
\end{equation}
and a factor of $i^{-\kappa} = \epsilon$ has been brought to the front. \\

The sums $G(m, n)$ and $G(x)$ are analogous to the sums $\sigma_{\psi,\epsilon}(m,n)$ and $\sigma_{\psi,\epsilon}(m,n;c)$ in the level aspect. Specifically, recall the sum
\begin{equation}
2\pi i^{-\kappa}\sum_{\substack{c > 0 \\ c\equiv 0\pmod q}}c^{-1}\sigma_{\psi,\epsilon}(m,n;c)J_{\kappa-1}\left(\frac{4\pi\sqrt{mn}}c\right)
\end{equation}
studied in Lemma \ref{cosetsumlemma}, where
\begin{equation}
\sigma_{\psi,\epsilon}(m,n;c) = \sum_{\chi\in\psi \widehat{G}_{p^j}}\chi(m)\overline{\chi}(n)\left[1+\epsilon\chi(-1)\right]S_{\overline{\chi}^2}(m,n;c).
\end{equation}
It is well known that the Bessel function mirrors the behavior of the Kloosterman sum in the level aspect. \\

Finally, we consider Lemma 5.8 of \cite{IW}, which states that
\begin{equation}
4G(x) = \epsilon g(x) - f(x) + O(xc_3(g)),
\end{equation}
where
\begin{equation}
f(x) = \int_0^{\infty}g\left(\sqrt{2xy}\right)\sin\left(x+y-\frac{\pi}4\right)\frac{dy}{\sqrt{\pi y}},
\end{equation}
\begin{equation}
g(x) = h\left(\frac{x-T+1}{\Delta}\right),
\end{equation}
and
\begin{equation}
c_3(g) = \int_{-\infty}^{\infty}\left|\widehat{g}(t)t^3\right|\;dt = \Delta\int_{-\infty}^{\infty}\left|\widehat{h}(\Delta t)t^3\right|\;dt = \frac1{\Delta^3}\int_{-\infty}^{\infty}|\widehat{h}(t)t^3|\;dt \ll \frac1{\Delta^3},
\end{equation}
with the implied constant depending only on $h$. Summing over $c$ on both sides gives
\begin{equation}
4G(m,n) = 2\pi\epsilon\sum_{c\ge 1}c^{-1}g\left(\frac{4\pi\sqrt{mn}}c\right)S(m,n;c) - 2\pi\sum_{c\ge 1}c^{-1}f\left(\frac{4\pi\sqrt{mn}}c\right)S(m,n;c) + O\left(\frac{\sqrt{mn}}{\Delta^3}\right),\label{weightsigmasum}
\end{equation}
where the implied constant is absolute. \\

The equation \eqref{weightsigmasum} is analogous to \eqref{sigmamnsum}. In particular, the first of the two sums over $c$ in \eqref{weightsigmasum} is analogous to the level aspect sum
\begin{equation}
\epsilon\psi(-1)\sum_{\substack{b\ge 1 \\ p\nmid b}}\frac{S(\overline{q}m,\overline{q}n;b)\mathcal{K}_{\psi}^-(\overline{b}m,\overline{b}n;q)}{bq}J_{\kappa-1}\left(\frac{4\pi\sqrt{mn}}{bq}\right) = \frac{\overline{\psi}(m)\psi(n)}{\varphi(p^j)}\sum_{\substack{c\ge 1 \\ p^k\mid\mid c}}c^{-1}\sigma_{\psi,\epsilon}(m,n;c)J_{\kappa-1}\left(\frac{4\pi\sqrt{mn}}c\right)
\end{equation}
in \eqref{sigmamnsum}. In the same way that the sum in the level aspect runs only over the $c$ that are divisible by $p$ exactly $k$ times (due to Lemma \ref{lemmakminus}), the sum in the weight aspect detects only the values of $c$ in a thin slice, since $\frac{4\pi\sqrt{mn}}c$ needs to be of size roughly $T$ -- recall that $T$ corresponds to $p^k$ in our analogy -- or else $g(x)$ vanishes. \\

Meanwhile, the second sum in \eqref{weightsigmasum} is analogous to the level aspect sum
\begin{equation}
\sum_{\substack{c> 0 \\ c\equiv 0\pmod{p^{j+k}}}}\frac{S_{\overline{\psi}^2}(m,n;c)}cJ_{\kappa-1}\left(\frac{4\pi\sqrt{mn}}c\right)
\end{equation}
in \eqref{sigmamnsum}. In the same way that the sum in the level aspect runs only over the $c$ that are divisible by $p$ at least $j+k$ times (due to Lemma \ref{lemmakplus}), we claim that the sum in the weight aspect runs only over the values of $c$ for which $\frac{4\pi\sqrt{mn}}c \gg T\Delta$ -- recall that $T\Delta$ corresponds to $p^{j+k}$ in our analogy -- up to an error term. Our proof of this again follows the presentation of \cite{IW}. First, we write
\begin{equation}
f(x) = \Im\left[\frac{e^{i\left(x-\frac{\pi}4\right)}}{\sqrt{\pi}}\int_0^{\infty}e^{iy}g\left(\sqrt{2xy}\right)\frac{dy}{\sqrt y}\right].
\end{equation}
Integrating by parts gives
\begin{equation}
f(x) = \Re\left[\frac{e^{i\left(x-\frac{\pi}4\right)}}{\sqrt{2\pi}}\int_0^{\infty}e^{iy}\frac{\sqrt{2xy}g'\left(\sqrt{2xy}\right) - g\left(\sqrt{2xy}\right)}{y^{3/2}}\;dy\right].
\end{equation}
Integrating by parts a second time gives
\begin{equation}
f(x) = -\Im\left[\frac{e^{i\left(x-\frac{\pi}4\right)}}{\sqrt{2\pi}}\int_0^{\infty}e^{iy}\frac{10xyg''\left(\sqrt{2xy}\right) - 15\sqrt{2xy}g'\left(\sqrt{2xy}\right) + 4g\left(\sqrt{2xy}\right)}{10y^{5/2}}\;dy\right].
\end{equation}
so
\begin{equation}
f(x) \ll \int_0^{\infty}\frac{2xy\left|g''\left(\sqrt{2xy}\right)\right|}{y^{5/2}}\;dy + \int_0^{\infty}\frac{\sqrt{2xy}\left|g'\left(\sqrt{2xy}\right)\right|}{y^{5/2}}\;dy + \int_0^{\infty}\frac{\left|g\left(\sqrt{2xy}\right)\right|}{y^{5/2}}\;dy,
\end{equation}
where the implied constant is absolute. Substituting $u=\sqrt{2xy}$ in each integral gives
\begin{equation}
f(x) \ll x^{3/2}\left(\int_0^{\infty}\frac{\left|g''(u)\right|}{u^2}\;du + \int_0^{\infty}\frac{\left|g'(u)\right|}{u^3}\;dy + \int_0^{\infty}\frac{\left|g(u)\right|}{u^4}\;du\right).
\end{equation}
Now supposing $h(x)$ has support contained in $(0, M)$, we see that $g(x)$ has support contained in $(T-1, T+\Delta M-1)$. Thus in general we have
\begin{equation}
\int_0^{\infty}\frac{\left|g^{(m)}(u)\right|}{u^{j+1}}\;du \le \int_{T-1}^{T+\Delta M-1}\frac{\left|\left|g^{(m)}\right|\right|_{\infty}}{u^{j+1}}\;du = \frac{\left|\left|h^{(m)}\right|\right|}{j\Delta^m}\left[\frac1{(T-1)^j} - \frac1{(T+\Delta M-1)^j}\right] \ll \frac1{T^{j+1}\Delta^{m-1}},
\end{equation}
where the implied constant depends only on $h$, $j$, and $m$. Thus we have
\begin{equation}
f(x) \ll \frac{x^{3/2}}{T^2\Delta},\label{fbound1}
\end{equation}
where the implied constant depends only on $h$. We also may write, without integrating by parts,
\begin{equation}
|f(x)| \le \frac1{\sqrt{\pi}}\int_0^{\infty}\frac{g\left(\sqrt{2xy}\right)}{\sqrt y}\;dy = \sqrt{\frac2{\pi x}}\int_0^{\infty}g(u)\;du = \sqrt{\frac2{\pi x}}\int_{T-1}^{T+\Delta M-1}||h||_{\infty}\;du = \frac{\sqrt2\Delta M||h||_{\infty}}{\sqrt{\pi x}},
\end{equation}
and thus
\begin{equation}
f(x) \ll \frac{\Delta}{x^{1/2}},\label{fbound2}
\end{equation}
where the implied constant depends only on $h$. Now if $x \ll T^2$, we use \eqref{fbound1} to get
\begin{equation}
f(x) \ll \frac{x^{3/2}}{T^2\Delta} \ll \frac{xT}{T^2\Delta} = \frac{x}{T\Delta}.
\end{equation}
Otherwise, if $x \gg T^2$, we use \eqref{fbound2} to get
\begin{equation}
f(x) \ll \frac{\Delta}{x^{1/2}} = \frac{x\Delta}{x^{3/2}} \ll \frac{x\Delta}{T^3} \ll \frac{x}{T\Delta}.
\end{equation}
Thus $f(x) \ll \frac{x}{T\Delta}$, so all the terms in the second sum of \eqref{weightsigmasum} where $\frac{4\pi\sqrt{mn}}c \ll T\Delta$ are bounded by
\begin{equation}
2\pi\sum_{c\gg \frac{4\pi\sqrt{mn}}{T\Delta}}c^{-1}\frac{\frac{4\pi\sqrt{mn}}c}{T\Delta}c^{\frac12+\varepsilon} \ll (mn)^{\frac14+\varepsilon}(T\Delta)^{-\frac12+\varepsilon} \ll \frac{\sqrt{mn}}{\Delta}.
\end{equation}
so the sum only runs over the $c$ for which $\frac{4\pi\sqrt{mn}}c\gg T\Delta$, up to an error term of $O\left(\frac{\sqrt{mn}}{\Delta}\right)$, where the implied constant depends only on $h$.

\end{section}

\begin{section}{Proof of Theorems \ref{thintheoremnv} and \ref{cosettheoremnv} (Nonvanishing Results)} \label{nvsection}

\stepcounter{subsection}

We start by proving Theorem \ref{cosettheoremnv}. Theorem \ref{thintheoremnv} will follow as a corollary. We note first that
\begin{equation}
\sum_{m\ge 0}p_m(\mathcal{F}, w) = 1.
\end{equation}
Now if we choose $\phi$ such that $\phi(0) = 1$ and $\phi(x) \ge 0$ for all $x$ (and as before, assume $\phi$ is even, Schwartz-class, and has a compactly supported Fourier transform) then we have
\begin{equation}
D_1(f,\phi, R) \ge \operatorname{ord}_{s=\frac12}L(s, f),
\end{equation}
just by ignoring all the terms except $\rho = \frac12$, which is counted according to its multiplicity. Then we have
\begin{equation}
\mathcal{D}_1(\mathcal{F},\phi,R,w) \ge \frac{\sum_{f\in\mathcal{F}}w(f)\operatorname{ord}_{s=\frac12}L(s, f)}{\sum_{f\in\mathcal{F}}w(f)} = \sum_{m\ge 1}mp_m(\mathcal{F}, w).
\end{equation}
In particular, for each positive integer $m$,
\begin{equation}
p_m(\mathcal{F}, w) \le \frac1m\mathcal{D}_1(\mathcal{F},\phi,R,w),
\end{equation}
by ignoring all but one term in the sum. Then suppose we have a result of the form
\begin{equation}
\lim_{q\to\infty}\mathcal{D}_1(\mathcal{F},\phi,R,w) = \mathcal{I}(\phi; G) = \int_{-\infty}^{\infty}\phi(x)W(G)(x)\;dx,
\end{equation}
for $\widehat{\phi}$ supported within $(-\theta, \theta)$, where $q$ is some parameter of the family (so that $\mathcal{F}$ and $R$ both depend on $q$) and $W(G)(x)$ is one of the group density distributions defined in \eqref{wux} through \eqref{wsoox}. For any $\varepsilon > 0$ and $m\ge 1$, we can write
\begin{align}
p_m(\mathcal{F}, w) &\le \frac1m\mathcal{I}(\phi; G) + \varepsilon, \\
p_0(\mathcal{F}, w) &\ge 1 - \mathcal{I}(\phi; G) - \varepsilon,
\end{align}
for sufficiently large $q$. If we assume that all $f\in\mathcal{F}$ have the same root number $\epsilon$, then we can come up with better lower bounds on the nonvanishing because we know that $p_m(\mathcal{F}, w) = 0$ unless $m$ has parity matching that of $\epsilon$. So we have
\begin{equation}
p_0(\mathcal{F}, w) \ge 1 - \frac12\mathcal{I}(\phi;G)-\varepsilon,\label{p0even}
\end{equation}
if $\epsilon = 1$, and
\begin{equation}
p_1(\mathcal{F}, w) \ge \frac32 - \frac12\mathcal{I}(\phi;G)-\varepsilon,\label{p1odd}
\end{equation}
if $\epsilon = -1$. (Note that $p_0(\mathcal{F}, w) = 0$ in this case; that is, $L(f, \frac12) = 0$ for every $f\in\mathcal{F}$.) \\

Now using the Fourier pair
\begin{align}
\phi(x) &= \left(\frac{\sin(\pi\theta x)}{\pi\theta x}\right)^2, \\
\widehat{\phi}(y) &= \frac1{\theta}\left(1-\frac{|y|}{\theta}\right),
\end{align}
given in (1.42) of \cite{ILS}, we can write the value of this integral in specific cases as follows:
\begin{align}
\mathcal{I}(\phi; \mathrm{U}) &= \frac1{\theta}, \label{ipu} \\
\mathcal{I}(\phi; \mathrm{O}) &= \frac1{\theta} + \frac12, \label{ipo} \\
\mathcal{I}(\phi; \mathrm{SO(even)}) &= \frac2{\theta} - \frac1{2\theta^2}, \label{ipsoe} \\
\mathcal{I}(\phi; \mathrm{SO(odd)}) &= 1 + \frac1{2\theta^2}, \label{ipsoo}
\end{align}
where the latter two are valid if $\theta \ge 1$. \\

In the case of the coset family, we have $\theta = 1 + \frac{j}k$ due to Theorem \ref{theorem2}. Substituting this into \eqref{ipsoe} and \eqref{ipsoo}, and applying either \eqref{p0even} or \eqref{p1odd}, gives the results \eqref{cosetevennv} and \eqref{cosetoddnv}, and this proves Theorem \ref{cosettheoremnv}. \\

To derive Theorem \ref{thintheoremnv} from this, we simply apply \eqref{p0even} and \eqref{p1odd} to $\mathcal{F}_{\kappa}(q,\chi,\eta)$, since all its elements have root number $\epsilon = i^{\kappa}\chi(-1)$, its symmetry group is $G = \mathrm{O}$, and we have support up to $\theta = 1$. Substituting this into \eqref{ipo}, and applying \eqref{p0even} and \eqref{p1odd}, gives the results \eqref{thinevennv} and \eqref{thinoddnv}.

\begin{remark}
We are not able to use this technique to get any results for non-quadratic characters $\eta$, since in this case the family only has unitary symmetry and the root number $\epsilon$ is not necessarily real.
\end{remark}

\end{section}

\begin{section}{Proof of Lemma \ref{weilavglemma} (The On-Average Weil Bound)} \label{weilsection}

\stepcounter{subsection}

To prove our on-average Weil bound, we first factor $c = p_1^{e_1}\cdots p_k^{e_k}$ and $\chi = \chi_1\cdots\chi_k$, where each $\chi_i$ is a Dirichlet character modulo $p_i^{e_i}$. Also, let $d_i$ be the inverse of $\frac{c}{p_i^{e_i}}$ modulo $p_i^{e_i}$ for each $i$. Then we can factor the Kloosterman sum to get
\begin{equation}
S_{\chi}(m,n;c) = \prod_{i=1}^kS_{\chi_i}(d_im,d_in;p_i^{e_i}).
\end{equation}
Next, for each $i$, define the residue $\ell_i$ modulo $p_i^{e_i}$ as follows. If $e_i$ is even, define it so that
\begin{equation}
\overline{\chi(1+zp_i^{e_i/2})} = e\left(\frac{\ell_iz}{p_i^{e_i/2}}\right),
\end{equation}
and if $e_i$ is odd, define it so that
\begin{equation}
\overline{\chi(1+zp_i^{(e_i-1)/2})} = e\left(\frac{\ell_iz}{p_i^{(e_i+1)/2}} + (p_i-1)\frac{\ell_iz^2}{2p_i}\right) = e\left(\frac{\ell_iz}{p_i^{(e_i+1)/2}} - \frac{\overline{2}\ell_iz^2}{p_i}\right),
\end{equation}
where the second equality holds unless $p_i = 2$. Furthermore, let $M_q(\alpha,\beta,\gamma)$ denote the number of solutions to the congruence $\alpha x^2+\beta x+\gamma \equiv 0\pmod q$. \\

Now we establish the following key lemma:
\begin{lemma}
For positive integers $m, n, c$, and $\chi$ a Dirichlet character modulo $c$, we have
\begin{equation}
|S_{\chi}(m,n;c)| \le c^{1/2}\tau(c)(m,n,\flrt(c))^{1/2}\left(\prod_{i=1}^kp_i^{\nu_{p_i}(c_i^2+4mn)}, \flrt(c)\right)^{1/2},\label{flrtlemma}
\end{equation}
where $c = p_1^{e_1}\cdots p_k^{e_k}$, $c_i = \frac{c\ell_i}{p_i^{e_i}}$, $\ell_i$ is as defined above, and $\flrt$ is as defined in \eqref{flrtdef}.
\end{lemma}

\begin{proof}
If $e_i$ is even, we have by the proof of Proposition 9.7 of \cite{KL} (which in turn uses Lemma 12.2 of \cite{IK}) that
\begin{equation}
S_{\chi}(d_im,d_in; p_i^{e_i}) = p_i^{e_i/2}\sum_{\substack{y\pmod{p_i^{e_i/2}} \\ p_i\nmid y \\ d_imy^2+\ell_iy-d_in\equiv 0\pmod{p_i^{e_i/2}}}}\overline{\chi(y)}e_{p_i^{e_i}}(d_imy+d_in\overline{y}),
\end{equation}
and thus the triangle inequality gives
\begin{equation}
|S_{\chi_i}(d_im,d_in;p_i^{e_i})| \le p^{e_i/2}M_{p_i^{e_i/2}}(d_im, \ell_i, -d_in) = p_i^{e_i/2}M_{p_i^{e_i/2}}(m, c_i, -n).
\end{equation}
If $e_i$ is odd and $e_i \ge 3$, we have by the proof of Proposition 9.8 of \cite{KL} (which in turn uses Lemma 12.3 of \cite{IK}) that
\begin{equation}
S_{\chi}(d_im,d_in; p_i^{e_i}) = p_i^{(e_i-1)/2}\sum_{\substack{y\pmod{p_i^{(e_i-1)/2}} \\ p_i\nmid y \\ d_imy^2+\ell_iy-d_in\equiv 0\pmod{p_i^{(e_i-1)/2}}}}\overline{\chi(y)}e_{p_i^{e_i}}(d_imy+d_in\overline{y})G_{p_i}(y),
\end{equation}
where $G_{p_i}(y)$ is the Gauss sum
\begin{equation}
G_{p_i}(y) = \sum_{z\pmod{p_i}}e_{p_i}\left(\left(d_iny^{-3} + (p_i-1)\frac{\ell_i}2y^{-2}\right)z^2 + \left(d_imy^2+\ell_iy-d_in\right)p_i^{-(e_i-1)/2}z\right).\label{gausssumodd}
\end{equation}
If the coefficient of $z^2$ does not vanish modulo $p_i$, the Gauss sum \eqref{gausssumodd} has absolute value $\sqrt{p_i}$, and so analogous to the case where $e_i$ is even, we have
\begin{equation}
|S_{\chi_i}(d_im,d_in;p_i^{e_i})| \le p^{e_i/2}M_{p_i^{(e_i-1)/2}}(d_im, \ell_i, -d_in) = p_i^{e_i/2}M_{p_i^{(e_i-1)/2}}(m, c_i, -n).
\end{equation}
Supposing instead that this coefficient vanishes, we must have
\begin{equation}
2d_in \equiv \ell_iy \pmod{p_i}.\label{z2vanish}
\end{equation}
Now $y$ also satisfies the congruence
\begin{equation}
d_imy^2 + \ell_iy - d_in \equiv 0\pmod{p_i}.\label{postnikovodd}
\end{equation}
We may multiply \eqref{postnikovodd} by $\ell_i^2$ and substitute in \eqref{z2vanish} to get the congruence
\begin{equation}
4d_i^3mn^2 + \ell_i^2d_in \equiv 0\pmod{p_i}.
\end{equation}
Now if $n\equiv 0\pmod{p_i}$, we also have $\ell_i\equiv 0\pmod{p_i}$ by \eqref{z2vanish}, and thus $m\equiv 0\pmod{p_i}$ by \eqref{postnikovodd}. Otherwise we can multiply by $n^{-1}d_i^{-3}$ to write this as
\begin{equation}
c_i^2 + 4mn \equiv 0 \pmod{p_i},
\end{equation}
and if $n\equiv 0\pmod{p_i}$, this is also necessarily true because $c_i\equiv 0\pmod{p_i}$. \\

Now if
\begin{equation}
\left(d_imy^2+\ell_iy-d_in\right)p_i^{-(e_i-1)/2} \equiv 0\pmod{p_i},
\end{equation}
then the Gauss sum \eqref{gausssumodd} equals $p_i$. Otherwise, it equals zero. Now suppose we take $y\pmod{p_i^{(e_i+1)/2}}$ and consider how many solutions to $my^2 + c_iy - n \equiv 0$ there are in some set of representatives modulo $p_i^{(e_i-1)/2}$. Since these representatives are evenly distributed across the equivalence classes, we see that there are precisely
\begin{equation}
\frac1{p_i}M_{p_i^{(e_i+1)/2}}(m, c_i, -n)
\end{equation}
such zeros. The remaining terms in the sum over $y$ vanish. Thus, in the case where the coefficient of $z^2$ in \eqref{gausssumodd} vanishes, we have
\begin{equation}
|S_{\chi_i}(d_im,d_in;p_i^{e_i})| \le p_i^{(e_i-1)/2}\cdot\frac1{p_i}M_{p_i^{(e_i+1)/2}}(m, c_i, -n)\cdot p_i = p_i^{(e_i-1)/2}M_{p_i^{(e_i+1)/2}}(m, c_i, -n).
\end{equation}
Finally, consider the case of $e_i=1$. In this case, the Kloosterman sum is over a prime modulus, so Proposition 9.4 of \cite{KL} applies directly. Thus
\begin{equation}
|S_{\chi_i}(d_im,d_in;p_i^{e_i})| \le 2p_i^{1/2}(m,n,p_i)^{1/2}.
\end{equation}

Next, we need to bound these $M_q(\alpha,\beta,\gamma)$ values. If $p\nmid\alpha$, we have
\begin{equation}
M_{p^e}(\alpha,\beta,\gamma) \le 2p^{\frac12\min\{e,\nu_p(\beta^2-4\alpha\gamma)\}},
\end{equation}
which follows from Lemma 9.6 of \cite{KL}. If $p\mid\alpha$, then unless $p\mid\beta$ and $p\mid\gamma$ as well, the result above still holds because either the discriminant is not divisible by $p$, or there are no roots. If $p \mid (\alpha,\beta,\gamma)$, then letting $\delta = \nu_p((\alpha,\beta,\gamma,p^e))$, we have
\begin{equation}
M_{p^e}(\alpha,\beta,\gamma) = p^{\delta}M_{p^{e-\delta}}\left(\frac{\alpha}{p^{\delta}}, \frac{\beta}{p^{\delta}}, \frac{\gamma}{p^{\delta}}\right) \le 2p^{\frac12\min\left\{e+\delta, \nu_p(\beta^2-4\alpha\gamma)\right\}}.
\end{equation}
Applying this bound, we see that in all the cases above, we have the bound
\begin{equation}
|S_{\chi_i}(d_im,d_in;p_i^{e_i})| \le 2p_i^{\frac{e_i+\delta_i}2}p_i^{\frac12\min\left\{\left\lfloor\frac{e_i}2\right\rfloor, \nu_{p_i}\left(c_i^2 + 4mn\right)\right\}},
\end{equation}
where $\delta_i = \nu_{p_i}\left(\left(m,n,\ell_i,p_i^{\lfloor\frac{e_i}2\rfloor}\right)\right)$. Multiplying together thus gives
\begin{equation}
|S_{\chi}(m,n;c)| \le \prod_{i=1}^k\left(2p_i^{\frac{e_i+\delta_i}2}p_i^{\frac12\min\left\{\left\lfloor\frac{e_i}2\right\rfloor, \nu_{p_i}\left(c_i^2 + 4mn\right)\right\}}\right).
\end{equation}
which simplifies to give \eqref{flrtlemma}.
\end{proof}

Now since $(m,n,\flrt(c)) \le (\flrt(mn), \flrt(c))$, we have
\begin{equation}
|S_{\chi}(m,n;c)| \le c^{1/2}\tau(c)(\flrt(r),\flrt(c))^{1/2}\left(\prod_{i=1}^kp_i^{\nu_{p_i}(c_i^2+4r)}, \flrt(c)\right)^{1/2}.
\end{equation}
where $r = mn$. Then we may sum up to write
\begin{equation}
\sum_{\substack{m,n\ge 1 \\ A\le mn\le B}}|S_{\chi}(m,n;c)| \le c^{1/2}\tau(c)\sum_{A\le r\le B}\tau(r)(\flrt(r),\flrt(c))^{1/2}\left(\prod_{i=1}^kp_i^{\nu_{p_i}(c_i^2+4r)}, \flrt(c)\right)^{1/2}.
\end{equation}
Applying the arithmetic-geometric mean inequality gives
\begin{equation}
\sum_{\substack{m,n\ge 1 \\ A\le mn\le B}}|S_{\chi}(m,n;c)| \le \frac12c^{1/2}\tau(c)\left[\sum_{A\le r\le B}\tau(r)(\flrt(r),\flrt(c)) + \sum_{A\le r\le B}\tau(r)\left(\prod_{i=1}^kp_i^{\nu_{p_i}(c_i^2+4r)}, \flrt(c)\right)\right].\label{amgm}
\end{equation}
It will suffice to show that both sums in \eqref{amgm} are bounded by $B^{\varepsilon}c^{\varepsilon}\left(B-A+c^{1/2}\right)$. Now we split up the first sum based on the gcd of $r$ and $c$, writing it as
\begin{equation}
\sum_{A\le r\le B}\tau(r)(\flrt(r),\flrt(c)) = \sum_{d\mid \flrt(c)}d\sum_{\substack{A\le r\le B \\ (\flrt(r),\flrt(c))=d}}\tau(r).
\end{equation}
We bound the inner sum as
\begin{equation}
\sum_{\substack{A\le r\le B \\ (\flrt(r),\flrt(c))=d}}\tau(r) \ll B^{\varepsilon}\left(\frac{B-A}{d^2} + 1\right).
\end{equation}
since we have $d^2\mid r$. Then we write
\begin{equation}
\sum_{A\le r\le B}\tau(r)(\flrt(r),\flrt(c)) \ll B^{\varepsilon}\sum_{d\mid\flrt(c)}\left(\frac{B-A}d + d\right) \ll B^{\varepsilon}c^{\varepsilon}\left(B-A+c^{1/2}\right).\label{amgmbound}
\end{equation}
We split up the second sum in \eqref{amgm} similarly, writing
\begin{equation}
\sum_{A\le r\le B}\tau(r)\left(\prod_{i=1}^kp_i^{\nu_{p_i}(c_i^2+4r)}, \flrt(c)\right) = \sum_{d\mid\flrt(c)}d\sum_{\substack{A\le r\le B \\ \left(\prod_{i=1}^kp_i^{\nu_{p_i}(c_i^2+4r)}, \flrt(c)\right)=d}}\tau(r).\label{funkysum}
\end{equation}
Now for each prime $p_i$ dividing $c$, we know that for $r$ in the inner sum, we have
\begin{equation}
\nu_{p_i}(c_i^2+4r) \ge \nu_{p_i}(d),
\end{equation}
and thus that
\begin{equation}
4r\equiv -c_i^2\pmod{p_i^{\nu_{p_i}(d)}}.
\end{equation}
If $p_i\neq 2$, there is a unique solution for $r$ modulo $p_i^{\nu_{p_i}(d)}$ since $4$ has an inverse modulo $p_i$. If $p_i = 2$, we claim there are at most 4 solutions for $r$ modulo $2^{\nu_2(d)}$. If $\nu_2(d) \le 2$ this is obvious, since there are at most 4 possible residues. If $\nu_2(d) > 2$ but $c_i$ is odd, there are no solutions. Otherwise, we divide the whole congruence by 4 to get
\begin{equation}
r\equiv -\left(\frac{c_i}2\right)^2\pmod{2^{\nu_2(d)-2}}.
\end{equation}
This gives a unique residue $r$ modulo $2^{\nu_2(d)-2}$, which can be lifted up to $2^{\nu_2(d)}$ in exactly four ways. Thus, putting all these residues together via the Chinese remainder theorem, we see that there are at most four possible residues for $r\pmod d$ for which $\left(\prod_{i=1}^kp_i^{\nu_{p_i}(c_i^2+4r)}, \flrt(c)\right)=d$. Thus the number of possible terms in the sum over $A\le r\le B$ satisfying this condition is at most $4\left(\frac{B-A}d+1\right)$, so the sum satisfies
\begin{equation}
\sum_{\substack{A\le r\le B \\ \left(\prod_{i=1}^kp_i^{\nu_{p_i}(c_i^2+4r)}, \flrt(c)\right)=d}}\tau(r) \ll B^{\varepsilon}\left(\frac{B-A}d+1\right).
\end{equation}
Then, as in \eqref{amgmbound}, we write
\begin{equation}
\sum_{A\le r\le B}\tau(r)\left(\prod_{i=1}^kp_i^{\nu_{p_i}(c_i^2+4r)}, \flrt(c)\right) \ll B^{\varepsilon}\sum_{d\mid\flrt(c)}(B-A+d) \ll B^{\varepsilon}c^{\varepsilon}\left(B-A+c^{1/2}\right),
\end{equation}
and this completes the proof of \eqref{weilavg}.

\end{section}

\bibliographystyle{alpha-abbrvsort}
\bibliography{refs}

\end{document}